\documentclass[12pt]{amsart}

\usepackage{amssymb,latexsym}
\usepackage[T1]{fontenc}
\usepackage{color}

\usepackage{enumerate}

 \usepackage[french,english]{babel}

\makeatletter

\@namedef{subjclassname@2010}{

  \textup{2010} Mathematics Subject Classification}

\makeatother
\newtheorem{thm}{Theorem}[section]

\newtheorem{lem}[thm]{Lemma}
\newtheorem{pro}[thm]{Proposition}
\theoremstyle{definition}

\newtheorem{rem}[thm]{Remark}

\numberwithin{equation}{section}

\newcommand{\X}{\mathbb{X}}
\newcommand{\Y}{\mathbb{Y}}

\newcommand{\ex}{\mathbb{E}}
\newcommand{\re}{\textup{Re}}

\newcommand{\pr}{\mathbb{P}}
\newcommand{\ep}{\varepsilon}

\newcommand{\F}{\mathcal{F}^*(x)}

\newcommand{\Fo}{\mathcal{F}(x)}

\newcommand{\newabstract}[1]{%
  \par\bigskip
  \csname otherlanguage*\endcsname{#1}%
  \csname captions#1\endcsname
  \item[\hskip\labelsep\scshape\abstractname.]
}

\begin{document}

\baselineskip=17pt

\title{The distribution of large quadratic character sums and applications}

\author{Youness Lamzouri}

\address{
Universit\'e de Lorraine, CNRS, IECL, F-54000 Nancy, FRANCE, 
and IRL3457 CRM-CNRS, Centre de Recherches Math\'ematiques, Universit\'e de Montr\'eal,
Pavillon Andr\'e-Aisenstadt,
2920 Chemin de la tour, 
Montr\'eal, QC, H3T 1J4, CANADA}

\email{youness.lamzouri@univ-lorraine.fr}

\date{}

\dedicatory{Dedicated to Andrew Granville on the occasion of his 60th birthday}

\begin{abstract} In this paper, we investigate the distribution of the maximum of character sums over the family of primitive quadratic characters attached to fundamental discriminants $|d|\leq x$. 
In particular, our work improves results of Montgomery and Vaughan, and gives strong evidence that the Omega result of Bateman and Chowla for quadratic character sums is optimal. We also obtain similar results for real characters with prime discriminants up to $x$, and deduce the interesting consequence that almost all primes with large Legendre symbol sums are congruent to $3$ modulo $4$. Our results are motivated by a recent work of Bober, Goldmakher, Granville and Koukoulopoulos, who proved similar results for the family of non-principal characters modulo a large prime. However, their method does not seem to generalize to other families of Dirichlet characters. Instead, we use a different and  more streamlined approach, which relies mainly on the quadratic large sieve. As an application, we consider a question of Montgomery concerning the positivity of sums of  Legendre symbols. 

\end{abstract}


\thanks{}

\maketitle

%%%%%%%%%%%%%%%%%%%%%%%%%%%%%%%%%%%%%%%%%%%%%%%%

\section{Introduction}

%%%%%%%%%%%%%%%%%%%%%%%%%%%%%%%%%%%%%%%%%%%%%%%%%%%%%%%%%%%%%%%%%%%%%%%%%%%%%%%%%%%%%%%%%%%%%%

Character sums play a central role in modern number theory through their numerous applications in the study of various arithmetic, analytic, algebraic and geometric objects. One important and basic example is that of quadratic Dirichlet characters, which include the Legendre symbol. The study of such characters and related sums
has a long and rich history stretching back to 
the work of Gauss on binary quadratic forms. Let $\chi$ be a Dirichlet character modulo $q$. One quantity that was extensively studied over the past century is 
$$ M(\chi):= \max_{t\leq q}  \Big|\sum_{n\leq t} \chi(n)\Big|.$$
In 1918, P\'{o}lya and Vinogradov independently proved that
$$M(\chi) \ll \sqrt{q} \log q.$$
On the other hand, an easy argument (based on applying Parseval's Theorem to the P\'olya Fourier series \eqref{PolyaFourier} attached to $\chi$) shows that $M(\chi)\gg \sqrt{q}$ for all primitive characters modulo $q$. 
Though one can establish the P\'olya-Vinogradov inequality using only basic Fourier analysis, improving on it has proved to be a very difficult problem, and resisted substantial progress outside of special cases. 
We should also note that any such improvement would have important consequences on several important quantities in analytic number theory, including class numbers of quadratic fields, short character sums, and the least quadratic non-residue (see for example the recent works \cite{BoGo}, \cite{GrMa} and \cite{Ma}). 
In 2007, Granville and Soundararajan \cite{GrSo07} made an important breakthrough by showing that the P\'olya-Vinogradov inequality can substantially be improved for characters of a fixed odd order. Further improvements to this case were obtained by Goldmakher \cite{Go}, and subsequently by Lamzouri and Mangerel \cite{LaMa}. 

Since $\sqrt{q}\ll M(\chi)\ll \sqrt{q}\log q$ for all primitive characters modulo $q$, a natural question is to determine the maximal order of $M(\chi)$. Assuming the Generalized Riemann Hypothesis GRH, Montgomery and Vaughan \cite{MV77} proved that $M(\chi)\ll \sqrt{q}\log\log q$. This turns out to be optimal (up to a constant factor) in view of an older result of Paley \cite{Pa} who proved the existence of an infinite family of quadratic characters for which $M(\chi)$ is that large. Granville and Soundararajan \cite{GrSo07} refined Montgomery and Vaughan's conditional result and showed that 
\begin{equation}\label{UpperGRH}
 M(\chi)\leq \Big(2 C_{\chi}+o(1)\Big) \sqrt{q}\log\log q,
 \end{equation}
for all primitive characters $\chi$ modulo $q$, where $C_{\chi}=e^{\gamma}/\pi$ if $\chi$ is odd, and $C_{\chi}=e^{\gamma}/(\sqrt{3}\pi)$ if $\chi$ is even. On the other hand, Paley's result was refined by Bateman and Chowla \cite{BaCh}, who proved the existence of an infinite sequence of moduli $q$, and primitive quadratic characters $\chi \bmod q$, such that 
\begin{equation}\label{LowerOmega}
M(\chi) \geq  \Big(\frac{e^{\gamma}}{\pi}+o(1)\Big) \sqrt{q}\log\log q.
\end{equation}
This result was extended to the family of primitive characters modulo a large prime $q$ by several authors (see for example Theorem 3 of \cite{GrSo07}). Finally, we should note that Granville and Soundararajan \cite{GrSo07} conjectured that Bateman and Chowla's Omega result should correspond to the true extreme values of $M(\chi)$, namely that 
\begin{equation}\label{ConjectureGrSo}
 M(\chi)\leq \Big(C_{\chi}+o(1)\Big) \sqrt{q}\log\log q, 
 \end{equation}
 for all primitive characters $\chi$. 
%Moreover, Granville and Soundararajan \cite{} proved that if $q$ is a large prime, then there are many primitive characters $\chi\bmod q$ such that 

\subsection{ The distribution of character sums} In view of \eqref{UpperGRH}, \eqref{LowerOmega} and \eqref{ConjectureGrSo} it is natural to renormalize $M(\chi)$ by defining 
$$ m(\chi):= \frac{e^{-\gamma} \pi}{\sqrt{q}} M(\chi).$$  
Montgomery and Vaughan \cite{MV79} were the first to study the distribution of $m(\chi)$ over families of Dirichlet characters. In particular, they showed that $m(\chi)$ is bounded (and hence $M(\chi)\ll \sqrt{q}$) for most characters. Let $q$ be a large prime and 
$$ \Phi_q(\tau):= \frac{1}{\varphi(q)} | \{\chi\neq \chi_0 \ (\bmod \ q) : m(\chi)>\tau\}|,$$
where $\varphi(q)$ is Euler's totient function. It follows from Montgomery and Vaughan's work \cite{MV79} that 
$$ \Phi_q(\tau) \ll_A \tau^{-A}, $$ for any constant $A\geq 1$. 
This estimate was improved by Bober and Goldmakher \cite{BoGo} for fixed $\tau$, and subsequently by Bober, Goldmakher, Granville and Koukoulopoulos \cite{BGGK} who showed that uniformly for $2\leq \tau \leq \log\log q-M$ (where $M\geq 4$ is a parameter) we have 
\begin{equation}\label{BGGK}
\exp\left(-\frac{e^{\tau+A_0-\eta}}{\tau}\big(1+O(E_1(\tau, M))\big)\right) \leq \Phi_q(\tau)\leq \exp\left(-\frac{e^{\tau-2-\eta}}{\tau}\big(1+O(E_2(\tau))\big)\right), 
\end{equation}
where $E_1(\tau, M)= (\log \tau)^2/\sqrt{\tau}+e^{-M/2}$, $E_2(\tau)=(\log \tau)/\tau$, $\eta:= e^{-\gamma}\log 2$, and $A_0=0.088546...$ is an explicit constant which can be expressed as a sum of integrals over the modified Bessel function of the first kind. In particular, this result gives strong evidence to the Granville-Soundararajan Conjecture \eqref{ConjectureGrSo} for odd primitive characters modulo a large prime $q$.

Although the family of quadratic characters was the first for which large character sums were exhibited by Paley \cite{Pa}, and then later by Bateman and Chowla \cite{BaCh}, no such distribution results are known in this case. In fact, the only known result for real characters is a result of Montgomery and Vaughan \cite{MV79} who showed that $\max_{t}\left|\sum_{n\leq t}\left(\frac{n}{p}\right)\right|\ll \sqrt{p}$ for most primes $p\leq x$. The main reason which explains why the results of Bober-Goldmakher \cite{BoGo} and Bober-Goldmakher-Granville-Koukoulopoulos \cite{BGGK} do not carry over to this setting is the fact that they rely heavily on the orthogonality relations for characters modulo $q$. 
%(which rely on the P\'olya-Vinogradov inequality and quadratic reciprocity, see for example Lemma ? of \cite{GrSo06}) are not precise, unlike the orthogonality relations for all non-principal characters modulo a large prime $q$ which are exact. 
Indeed, the key ingredient in the proof of \eqref{BGGK} is to estimate the off-diagonal terms when bounding large moments of the tail of the sum in P\'olya's Fourier expansion \eqref{PolyaFourier} below, which the authors of \cite{BGGK} successfully achieved using intricate estimates involving divisor functions.

%This was successfully carried out for the family of all non-principal characters modulo a large prime $q$ by Bober, Goldmakher, Granville and Koukoulopoulos \cite{BGGK} 

In this paper, we overcome this problem for the family of quadratic characters by using a different and more streamlined approach which relies on the quadratic large sieve. Before stating our results we need some notation. For a fundamental discriminant $d$ we let $\chi_d(\cdot)= \left(\frac{d}{\cdot}\right)$ be the Kronecker symbol modulo $|d|$. It is useful here to consider the cases of positive and negative discriminants separately, since as Theorems \ref{Main} and \ref{MainPositive} below show, the distribution of large values of $m(\chi_d)$ behaves differently in each case. The difference between these cases lies in the fact that the character $\chi_d$ is even if $d$ is positive, and is odd if $d$ is negative. Thus, in view of Conjecture  \eqref{ConjectureGrSo} we expect the extreme values of $m(\chi_d)$ for $d>0$ to be smaller by a factor of $\sqrt{3}$ compared to the case $d<0$. We denote by $\mathcal{F}(x)$ the set of fundamental discriminants $d$ such that $|d|\leq x$, and let $\mathcal{F}^{+}(x)$ (respectively $\mathcal{F}^{-}(x)$) be the subset of $\mathcal{F}(x)$ consisting of positive (respectively negative) discriminants. Then we have the following standard estimates (see for example Lemma 4.1 of \cite{GrSo06})
$$  |\mathcal{F}^{\pm}(x)|= \frac{3}{\pi^2}x +O_{\ep}(x^{1/2+\ep}).$$
Our goal is to estimate the distribution functions $$\Psi^{\pm}_x(\tau):=\frac{1}{|\mathcal{F}^{\pm}(x)|}  |\{d\in \mathcal{F}^{\pm}(x) : m(\chi_d)>\tau\}|,$$
uniformly for $\tau$ in the range $2\leq \tau\leq (1+o(1))\log\log x$ in the case of  $\Psi^{-}_x(\tau)$, and $2\leq \tau\leq (1/\sqrt{3}+o(1))\log\log x$ in the case of  $\Psi^{+}_x(\tau)$. Conjecture \eqref{ConjectureGrSo} implies that these ranges are best possible up to the term $o(\log\log x)$. Here and throughout we shall denote by $\log_k$ the $k$-th iteration of the natural logarithm. We prove the following results.
\begin{thm}\label{Main}
Let $\eta= e^{-\gamma}\log 2$, and $x$ be a large real number. Uniformly for $\tau$ in the range $2\leq \tau \leq \log_2 x+\log_5 x-\log_4 x- C$ (where $C>0$ is a suitably large constant) we have 
$$ \exp\left(-\frac{e^{\tau-\eta
-B_0}}{ \tau} \left(1+O\left(\frac{(\log \tau)^2 }{\sqrt{\tau}}\right)\right)\right)\leq  \Psi^{-}_x(\tau)\leq \exp\left(-\frac{e^{\tau-\eta -\log 2 - 2}}{ \tau} \left(1+O\left(\frac{\log \tau }{\tau}\right)\right)\right),$$
where 
\begin{equation}
B_0= \int_0^1 \frac{\tanh y}{y} dy + \int_1^{\infty} \frac{\tanh y-1}{y}dy=0.8187...
\end{equation}
\end{thm}
\begin{thm}\label{MainPositive}
Let $B_0$ be the constant in Theorem \ref{Main}. There exists positive constants $C_1$ and $C_2$ such that uniformly for $\tau$ in the range $2\leq \tau \leq (\log_2 x+\log_5 x-\log_4 x- C_1)/\sqrt{3}$ we have 
$$ \exp\left(-\frac{e^{\sqrt{3}\tau-B_0}}{\sqrt{3}\tau} \left(1+O\left(\frac{1 }{\tau}\right)\right)\right)\leq \Psi^+_x(\tau)\ll \exp\left(-\frac{e^{\sqrt{3}\tau}}{ \tau^{C_2}}\right).$$
\end{thm}
\begin{rem} Theorems \ref{Main} and \ref{MainPositive} (the latter is the analogue of Theorem 1.3 of \cite{BGGK}) give strong evidence to the Granville-Soundararajan Conjecture \eqref{ConjectureGrSo} for the family of quadratic characters. Moreover, a direct consequence of these results is the fact that almost all fundamental discriminants $|d|\leq x$ for which $M(\chi_d)$ is large are negative.
\end{rem}

\begin{rem} The lower bounds in \eqref{BGGK} and Theorems \ref{Main} and \ref{MainPositive} are consequences of the works of Granville and Soundararajan \cite{GrSo03} and \cite{GrSo06} on the distribution of $|L(1, \chi)|$ (over non-principal characters modulo $q$) in the case of \eqref{BGGK}, and $L(1, \chi_d)$ (over fundamental discriminants $|d|\leq x$) in the case of Theorems \ref{Main} and \ref{MainPositive} (see Section 2 for more details). Moreover, note that we obtain a slightly different constant in the upper bound of Theorem \ref{Main} compared to \eqref{BGGK}. We believe that this is caused by the different nature of the family of quadratic characters. Indeed in our case, the difference between the constants in the upper and lower bounds of Theorem \ref{Main} is $2+\log 2- B_0\approx 1.8744$, which is slightly smaller than the analogous difference  $2+A_0\approx 2.0885$ for the family of non-principal characters modulo a large prime $q$ in \eqref{BGGK}.
Finally, we note that by using our approach, we can derive an easier proof of the upper bound of \eqref{BGGK} in the range $2\leq \tau\leq \log_2 q+\log_5 q-\log_4 q- C$. This is achieved by following the exact same argument of the proof of the upper bound of Theorem \ref{Main}, and replacing the quadratic large sieve inequalities of Heath-Brown and Elliott (see Lemma \ref{LargeSieve} below) by the following large sieve estimate of Montgomery (see Theorem 6.2 of \cite{MoBook}) 
$$\sum_{\chi\bmod q} \left|\sum_{n\leq N} a_n\chi(n)\right|^{2}\ll (q+N)\sum_{n\leq N} |a_n|^2,$$
which hold for an arbitrary complex sequence $\{a_n\}_{n\geq 1}$, and all integers $q\geq 2$. Although our approach gives a slightly smaller range of $\tau$ in this case, it has the advantage of extending the upper bound of \eqref{BGGK} to all moduli $q$. 
\end{rem}

%\begin{cor}\label{CorPositive}
%Let $\tau \leq \log_2 x+\log_5 x-\log_4 x- C$ (where %$C>0$ is a suitably large constant) be a large real %number. Then we have $ \Psi_x^{+}(\tau)= %o(\Psi_x(\tau))$. \end{cor}

\begin{rem}
Using our work we can show that the structure results for large character sums obtained by Bober, Goldmakher, Granville and Koukoulopoulos in Section 2 of \cite{BGGK} for the family of non-principal characters modulo a large prime $q$,  hold verbatim for the family of quadratic characters $\chi_d$ attached to fundamental discriminants $|d|\leq x$. Since these results are technical, and the statements are exactly the same, we prefer to not state them here and refer the reader to the exact statements in \cite{BGGK}. The proofs follow along the  same lines of \cite{BGGK} %(which rely principally on the ``pretentious'' theory of character sums developed by Granville and Soundararajan in \cite{GrSo07}), 
by using the auxiliary lemmas therein which hold for all primitive characters (and are derived using the ``pretentious'' theory of character sums developed by Granville and Soundararajan in \cite{GrSo07}), and  replacing the ingredients of the proof of Theorem 1.1 in \cite{BGGK} by those of the proof of Theorem \ref{Main} of our paper.
\end{rem}

\subsection{Analogous results for prime discriminants} Using our approach we establish similar results to Theorems \ref{Main} and \ref{MainPositive} over prime discriminants. The analogous lower bounds are direct consequences of newly established results on the distribution of $L(1, (\frac{\cdot}{p}))$, which we shall describe in the next section. %However our lower bound for the distribution function is less precise in this case, due to the possible existence of Landau-Siegel zeros. 
Furthermore, the analogous upper bounds will be obtained using the same methods of proofs of Theorems \ref{Main} and \ref{MainPositive}, together with the large sieve inequality of Montgomery and Vaughan \cite{MV79} for prime discriminants (see Lemma \ref{LargeSieveMV} below). Since the Legendre symbol modulo $p$ is even if $p\equiv 1\bmod 4$, and is odd if $p\equiv 3 \bmod 4$, we shall consider the cases of primes congruent to $1$ and $3$ modulo $4$ separately. 
For $a\in \{1, 3\}$, let 
$\Psi_{x, a}^{\textup{ prime }}(\tau)$ be the proportion of primes $p\leq x$ such that $p\equiv a \bmod 4$ and $m\left(\left(\frac{\cdot}{p}\right)\right)>\tau$. 
\begin{thm}\label{MainPrimes} Let $\eta$ and $B_0$ be the constants in Theorem \ref{Main}. There exists positive constants $C_1, C_2$ such that
\begin{itemize}
\item[1.] Uniformly in the range 
$2\leq \tau \leq \log_2 x+\log_5 x-\log_4 x- C_1$ we have
$$
\Psi_{x, 3}^{\textup{ prime }}(\tau)\leq \exp\left(-\frac{e^{\tau-\eta -\log 2 - 2}}{ \tau} \left(1+O\left(\frac{\log \tau }{\tau}\right)\right)\right).
$$
\item[2.] Uniformly in the range $2\leq \tau \leq (\log_2 x)/2-2 \log_3 x$ we have
\begin{equation}\label{LowerBPrimeDis}
\Psi_{x, 3}^{\textup{ prime }}(\tau) \geq \exp\left(-\frac{e^{\tau-\eta
-B_0}}{ \tau} (1+o(1))\right).
\end{equation}
\item[3.] And uniformly in the range $(\log_2 x)/2-2 \log_3 x \leq \tau \leq \log_2x-\log_3 x-C_1$ we have
$$
\Psi_{x, 3}^{\textup{ prime }}(\tau) \geq \exp\left(-C_2\tau e^{\tau} \right).
$$
\end{itemize}
\end{thm} 
\begin{thm}\label{MainPositivePrimes} Let  $B_0$ be the constant in Theorem \ref{Main}. There exists positive constants $C_1, C_2$ and $C_3$ such that
\begin{itemize}
\item[1.] Uniformly in the range 
$2\leq \tau \leq (\log_2 x-\log_3 x- C_1)/\sqrt{3}$ we have
$$
\exp\left(-C_3\tau e^{\sqrt{3}\tau} \right)\leq \Psi_{x, 1}^{\textup{ prime }}(\tau)\ll \exp\left(-\frac{e^{\sqrt{3}\tau}}{ \tau^{C_2}}\right).
$$
\item[2.] Moreover, uniformly in the smaller range $2\leq \tau \leq \big((\log_2 x)/2-2 \log_3 x\big)/\sqrt{3}$ we have the improved lower bound
\begin{equation}\label{LowerBPositivePrimeDis}
\Psi_{x, 1}^{\textup{ prime }}(\tau) \geq \exp\left(-\frac{e^{\sqrt{3}\tau
-B_0}}{\sqrt{3}\tau} (1+o(1))\right).
\end{equation}
\end{itemize}
\end{thm} 
\begin{rem}\label{PreciseLowerConditional}
An interesting consequence of Theorems \ref{MainPrimes} and \ref{MainPositivePrimes} is that almost all  primes $p\leq x$ for which $M\Big(\Big(\frac{\cdot}{p}\Big)\Big)$ is large are congruent to $3$ modulo $4$. Moreover, we note that conditionally on the GRH, the lower bound \eqref{LowerBPrimeDis} (respectively \eqref{LowerBPositivePrimeDis}) holds in the extended range $2\leq \tau \leq \log_2 x-2 \log_3 x-C_4$ (respectively $2\leq \tau \leq (\log_2 x-2 \log_3 x-C_4)/\sqrt{3}$), for some positive constant $C_4$. See Remark \ref{PreciseLowerConditional2} below for a justification of this fact. 
\end{rem}

\begin{rem}\label{LimitingDistributionPrime}
In Theorem 1.4 of \cite{BGGK}, Bober, Goldmakher, Granville and Koukoulopoulos establish the existence of a limiting distribution for $m(\chi)$, as $\chi$ varies over non-principal characters modulo a large prime $q$, when $q\to \infty$. In a recent joint work with Hussain \cite{HuLa}, we establish an analogous result for 
$M\Big(\Big(\frac{\cdot}{p}\Big)\Big)$ as $p$ varies over the primes in a large dyadic interval $[Q, 2Q]$ and $Q\to \infty$. 
\end{rem}

\subsection{The distribution of $L(1, (\frac{\cdot}{p}))$}
In order to prove the lower bounds of Theorems \ref{MainPrimes} and \ref{MainPositivePrimes} we need estimates on the distribution of $L(1, (\frac{\cdot}{p}))$, as $p$ varies over the primes. It turns out that this case is harder than the case of the bigger family of primitive quadratic characters attached to fundamental discriminants $d$, which was  investigated by Granville and Soundararajan in \cite{GrSo03} (see the precise statement of their result in Theorem \ref{GrSoThm} below). By the law of quadratic reciprocity, this is due to the fact that current bounds on character sums over primes are much weaker than analogous bounds for character sums over the integers.\footnote{Bounds for character sums over primes are ultimately connected to the distribution of primes in arithmetic progressions. The weakness of such bounds explains for example why the strongest version of the Prime Number Theorem for arithmetic progressions, namely the Siegel-Walfisz Theorem, only holds for very small moduli.}  In particular, such bounds are heavily affected by the possible existence of a Landau-Siegel exceptional discriminant. 

In \cite{Jo} Joshi  extended Littlewood's Omega result \cite{Li}, by establishing the existence of infinitely many primes $p$ such that 
$$ L(1, (\frac{\cdot}{p}))\geq (e^{\gamma}+o(1))\log\log p.$$
We improve on this result by obtaining estimates for the proportion of primes $p\leq x$ such that $L(1, (\frac{\cdot}{p}))>e^{\gamma}\tau$ uniformly for $\tau$ in the range $2\leq \tau\leq (1-o(1))\log\log p$, which is believed to be best possible up to the factor $o(\log\log p)$ (see for example the conjectures of Montgomery and Vaughan \cite{MV99} and Granville and Soundararajan \cite{GrSo03}). For $\tau>0$ and $a\in \{1, 3\}$ we define 
$$ 
F_{x, a}(\tau):= \frac{2}{\pi(x)} \left|\left\{p\leq x :  p\equiv a \bmod 4, \ L(1, (\frac{\cdot}{p}))>e^{\gamma} \tau\right\}
\right|.
$$
\begin{thm}\label{LamzouriL1chi}
Let $a\in \{1, 3\}$ and $x$ be large. In the range $2\leq \tau\leq (\log_2 x)/2-2\log_3 x$ we have 
\begin{equation}\label{EstimateDistribL1psi} 
F_{x, a}(\tau)= \exp\left(-\frac{e^{\tau-B_0}}{ \tau} \left(1+o(1)\right)\right).
\end{equation}
Moreover, there exists positive constants $C_1, C_2>0$ such that in the range $(\log_2 x)/2-2\log_3 x\leq \tau\leq \log_2 x-\log_3x-C_2$ we have 
\begin{equation}\label{EstimateDistribL1psi2} 
\exp\left(-C_1\tau e^{\tau} \right)\leq F_{x, a}(\tau) \leq \exp\left(-\frac{e^{\tau+\log 2-2}}{ \tau}\left(1+O\left(\frac{\log \tau}{\tau}\right)\right)\right). 
\end{equation}
Furthermore, the same estimates hold for the proportion of primes $p\leq x$ such that $p\equiv a \bmod 4$ and $L(1, (\frac{\cdot}{p})\chi_{-3})>(2e^{\gamma}/3)\tau$, where $\chi_{-3}$ is the non-principal character modulo $3$. 
\end{thm}
\begin{rem}
 One can compare our results with those of Granville and Soundararajan \cite{GrSo03} for the distribution of $L(1, \chi_d)$ over fundamental discriminants $|d|\leq x$ (see Theorem \ref{GrSoThm} below). In this case, the same estimate \eqref{EstimateDistribL1psi} holds for the proportion of fundamental discriminants $|d|\leq x$ such that $L(1, \chi_d)>e^{\gamma} \tau$, uniformly for $\tau$ in the larger  range $2\leq \tau\leq \log_2 x + (1-o(1)) \log_4 x $. We should also note that conditionally on GRH, the estimate \eqref{EstimateDistribL1psi} holds in the extended range $2\leq \tau \leq \log_2 x-2 \log_3 x-C_2$ (for some positive constant $C_2$) by a result of   Holmin,  Jones, Kurlberg,  McLeman and Petersen \cite{HJKMP} (see Remark \ref{PreciseLowerConditional2} below). Finally, we note that in the case $a=3$, our proof yields a better constant in the upper bound of \eqref{EstimateDistribL1psi2} than if we just apply Theorem \ref{MainPrimes} directly (using the inequality \eqref{LowerBoundML} below). 
\end{rem}

To prove the precise estimate \eqref{EstimateDistribL1psi} in the smaller range $2\leq \tau\leq (1/2-o(1))\log_2x$, we use our previous work \cite{La17}, where we established asymptotic formulas for the complex moments of $L(1, (\frac{\cdot}{p}))$ involving a secondary term which is coming from a possible Landau-Siegel exceptional discriminant.
%(see Section 5 for the proper definition). 
Although we could not rule out that this term might be as large as the main term, we were able to show that it does not affect the leading term in the tail of the distribution of $L(1, (\frac{\cdot}{p}))$. However, we note that the error term $o(1)$ inside the estimate \eqref{EstimateDistribL1psi} is not effective, due to the use of Siegel's Theorem. The upper bound of \eqref{EstimateDistribL1psi2} will be a consequence of Theorem \ref{KeyResultPrimes} below, which is the key ingredient in the proofs of the upper bounds of Theorems \ref{MainPrimes} and \ref{MainPositivePrimes}. Finally, to obtain the lower bound of \eqref{EstimateDistribL1psi2}, we combine Theorem \ref{KeyResultPrimes} with a strong form of Linnik's Theorem established by Bombieri \cite{Bo} using his zero density estimates for Dirichlet $L$-functions. 

\subsection{An application to a question of Montgomery} 

As an application of our results we consider the problem of the positivity of sums of the Legendre symbol. Let $p\ge 3$ be a prime and 
$$ S_p(t)=\sum_{n\leq t} \left(\frac{n}{p}\right). $$
The question of determining when  $S_p(t)$ is positive was considered by several mathematicians including Fekete, Chowla, Montgomery and others. Since $S_p(t)$ is periodic of period $p$, one can renormalize the variable $t$ and define for $\alpha \geq 0$ the function
$$ f_p(\alpha):= \sum_{0\leq n\leq \alpha p} \left(\frac{n}{p}\right), $$
which is periodic of period $1$. One can extend $f_p$ to all $\alpha \in \mathbb{R}$ by periodicity. 
In \cite{Mo} Montgomery studied the following natural question: How frequently is $f_p(\alpha)$ positive for a prime $p\equiv 3 \bmod 4$?  More precisely, he investigated the quantity
$$ 
\lambda(p):=\mu\left(\left\{\alpha \in [0, 1) : f_p(\alpha)>0\right\}\right),
$$
for such primes\footnote{Note that $f_p$ is even if $p\equiv 3 \bmod 4$ and is odd if $p\equiv 1 \bmod 4$. The analogous question in the latter case becomes: how often is $f_p(\alpha)>0$ for $0<\alpha<1/2$? We shall only consider the former case in this paper since our methods can be extended to handle the case of primes $p\equiv 1 \bmod 4$. }, where $\mu(\mathcal{C})$ denotes the Lebesgue measure of a measurable set $\mathcal{C}$. 
He proved in \cite{Mo} that for all primes $p\equiv 3 \bmod 4$ we have $\lambda(p)>1/50$. He also established the existence of infinitely many such primes such that $\lambda(p)<1/3+\varepsilon$ for any fixed $\varepsilon>0$. In her Master's thesis \cite{Me}, Mehkari slightly improved the constant $1/50$ in the first result of Montgomery. She also run extensive numerical computations which suggest that the value $1/3$ is his second result is optimal. Furthermore, she proved conditionally on GRH that $\lambda(p)\leq 0.764$ for a positive proportion of the primes $p\equiv 3 \bmod 4$, and  that $\lambda(p)\geq 0.285$ for a positive proportion of the primes $p\equiv 3 \bmod 4$. 
Montgomery also writes in \cite{Mo}  ``the ideas found in our proof can also be used to show that there are infinitely many primes $p\equiv 3 \bmod 4$ such that $\lambda(p)>1-\varepsilon$." Using a different method, based on the proof of Theorem \ref{MainPrimes}, we improve on these results by showing that for any $\varepsilon>0$, both inequalities $\lambda(p)>1-\varepsilon$ and $\lambda(p)<1/3+\varepsilon$ hold for a positive proportion of the primes $p\equiv 3 \bmod 4$. We also quantify these proportions in terms of $\varepsilon$ and consider the question of uniformity by letting $\varepsilon \to 0$ slowly as  a function of $x$, if we vary over the primes $p\leq x$ such that $p\equiv 3\bmod 4$. 
\begin{thm}\label{PositivityLegendreTheorem} Let $\nu>0$ be a small fixed constant. Let $x$ be large and $1<T \leq \exp\big((1-\nu)\log_2 x\log_3 x/(\log_4 x)\big)$ be a real number. The number of primes $p\leq x$ with $p\equiv 3\bmod 4$ and such that $\lambda(p)>1-1/T$ is
\begin{equation}\label{PreciseLargeLambda}
\gg \pi(x)\exp\left(-\exp\left(\frac{\log T\log_3 T}{\log_2T}(1+o(1))\right)\right).
\end{equation}
In particular,  this quantity is 
$$\gg_{\nu} \pi(x)\exp\left(-T^{\nu}\right).$$
\end{thm}
\begin{thm}\label{NegativityLegendreTheorem} Let $x$ be large and $T>1$. There exists positive constants $c_1, c_2, c_3$ such that the number of primes $p\leq x$ with $p\equiv 3\bmod 4$ and such that $\lambda(p)<1/3+1/T$ is at least 
\begin{equation}\label{PreciseSmallLambda}
\pi(x)\exp\left(-c_1T^2(\log T)^3\right),
\end{equation}
if $1<T \leq c_2 (\log_2 x)^{1/2}/(\log_3 x)^{2}$, and is at least \begin{equation}\label{SmallLambda}
\pi(x)\exp\left(-c_1T^2 (\log T)^5\right),
\end{equation}
if $c_2 (\log_2 x)^{1/2}/(\log_3 x)^{2}\leq T\leq c_3 (\log x)^{1/2}/(\log_2 x)^{5/2}$. 
\end{thm}

\subsection{Further applications} We should note that our method could be adapted to investigate the distribution of $M(\chi)$ over the family of primitive cubic Dirichlet characters with conductor up to $x$. We are planning to pursue this direction in a future paper. Moreover, we mention that the ideas of the proof of Theorem \ref{MainPrimes} are used in a different work of the author, joint with Hussain \cite{HuLa}, concerning the limiting distribution of character paths attached to the family of Legendre symbols modulo primes. Hussain \cite{Hu} previously established a similar result for character paths attached to non-principal characters modulo a large prime $q$. Furthermore, in her PhD thesis \cite{HuT}, she proved conditionally on GRH that character paths attached to the family of Legendre symbols converge in law (in the space of continuous functions) to a random Fourier series constructed using Rademacher random multiplicative functions. Our forthcoming work will establish this result unconditionally.

\subsection{Acknowledgments} The author would like to thank Andrew Granville for useful discussions. We would also like to thank the anonymous referee for carefully reading the paper and for their comments and suggestions. This work was completed while the author was on a D\'el\'egation CNRS  at the IRL3457 CRM-CNRS in Montr\'eal. We would like to thank the CNRS for its support and the Centre de Recherches Math\'ematiques for its excellent working conditions. 

%%%%%%%%%%%%%%%%%%%%%%%%%%%%%%%%%%%%%%%%%%%%%%%%%%%%%%%%%%%%%%%%%%%%%%%%%%%%%%%%%%%%%%%%%%%%%%%%%%%%%%%%%%%%%%%%%%%%%%%%%%%%%%

\section{Outline and key ingredients of the proofs of Theorems \ref{Main}, \ref{MainPositive} and  \ref{MainPrimes}}

\subsection{The overall strategy of the proof of Theorem \ref{Main}}

We first start by describing the strategy and key ideas of the proof of Theorem \ref{Main}, since the method for proving the other results Theorems \ref{MainPositive}, \ref{MainPrimes} and \ref{MainPositivePrimes} is similar. In particular, it will be useful to compare the argument with \cite{BGGK}. The character sum $\sum_{n\leq t} \chi_d(n)$ has a simple Fourier expansion first obtained by P\'olya in  the following quantitative form (eq. (9.19), p.311 of \cite{MVBook}) 
\begin{equation}\label{PolyaFourier}
\sum_{n\leq t} \chi_d(n)= \frac{\mathcal{G}(\chi_d)}{2\pi i}\sum_{1\leq |n|\leq Z} \frac{\chi_d(n) (1-e(-nt/|d|))}{n} +O\left(1+ \frac{|d|\log|d|}{Z}\right),
\end{equation}
where $\mathcal{G}(\chi_d)$ is the Gauss sum attached to $\chi_d$. Let $\delta:=1/100$, and define $\F$ to be the set of fundamental discriminants $x^{1-\delta}\leq |d|\leq x$. We note that the proportion of those discriminants $d$ with $|d|\leq x^{1-\delta}$ is $\ll x^{-\delta}$, which is much smaller than the distribution function $\Psi^-_x(\tau)$ in the range $\tau\leq \log\log x$. Hence, we will only focus on the fundamental discriminants $d\in \F$. 
For these discriminants we have
\begin{equation}\label{Polya}
m(\chi_d)= \frac{e^{-\gamma}}{2} \max_{\alpha \in [0, 1)} \left|\sum_{1\leq |n|\leq Z} \frac{\chi_d(n) (1-e(\alpha n))}{n}\right| + O\left(x^{-\delta}\right),
\end{equation}
where $Z=x^{21/40}$. 
%Therefore, in order to study the distribution of $m(\chi_d)$ it suffices to investigate the distribution of the series 
%$$ \sum_{n\in \mathbb{Z}\setminus\{0\}} \frac{\chi_d(n) (1-e(\alpha n))}{n}.$$
We first describe a heuristic argument that will help us isolate the key ingredient in the proof of Theorem \ref{Main}. This will also explain why we should expect the tail of the distribution of $m(\chi_d)$ to behave like $\exp(-e^{\tau}/\tau)$. A standard idea, which goes back to the work of Montgomery and Vaughan \cite{MV77} on exponential sums with multiplicative coefficients, is to split the sum on the right hand side of \eqref{Polya} into two parts
\begin{equation}\label{SeriesPolya}\sum_{1\leq |n|\leq Z} \frac{\chi_d(n) (1-e(\alpha n))}{n}= \sum_{\substack{1\leq |n|\leq Z\\ P^{+}(n)\leq y}} \frac{\chi_d(n) (1-e(\alpha n))}{n}+\sum_{\substack{1\leq |n|\leq Z\\ P^{+}(n)> y}}\frac{\chi_d(n) (1-e(\alpha n))}{n}, 
\end{equation}
and show that uniformly over $\alpha$, the bulk of the distribution comes from the first part, over $y$-friable integers, with a suitable choice of the parameter $y$. Here $P^{+}(n)$ is the largest prime factor of $n$, and an integer $n$ is called $y$-friable (or $y$-smooth) if $P^+(n)\leq y$. To understand what choice of the parameter $y$ we should make we observe that 
\begin{equation}\label{BoundTrivialYFriablePart}
\max_{\alpha\in [0, 1)} \Bigg|\sum_{\substack{1\leq |n|\leq Z\\ P^{+}(n)\leq y}} \frac{\chi_d(n) (1-e(\alpha n))}{n} \Bigg|\ll \sum_{\substack{n\ge 1\\ P^{+}(n)\leq y}}\frac{1}{n}\ll \log y
\end{equation} 
by Mertens' Theorem. Hence, if the main part of the contribution to $m(\chi_d)>\tau$ is coming from $y$-friable integers, we should aim for a choice of $y$ such that\footnote{This is not completely correct, since the above argument shows that we rather have $y\approx e^{c\tau}$ for some constant $c>0$. However, it turns out that the optimal choice of $c$ is $c=1$.} $y\approx e^{\tau}$. Heuristically, for small $y$  ($y\leq \log x$ say) we can prescribe the values of $\chi_d(p)$ for $p\leq y$ with probability\footnote{This is correct in the range $y\leq \log\log x$, as proved in Lemma \ref{SignsLegendre} below. It also follows from GRH in the larger range $y\leq c\log x\log\log x$ for some small constant $c>0$, see Theorem 13.5 of \cite{MoBook}.} $2^{-\pi(y)}=\exp(-(\log 2+o(1)) y/\log y)$ by the Prime Number Theorem. Therefore, if $y\asymp e^{\tau}$ this probability looks like $\exp(-c_1e^{\tau}/\tau)$ for some positive constant $c_1$, which agrees with the statement of Theorem \ref{Main}. Thus, in order for this heuristic to work, we need to efficiently control the second part in \eqref{SeriesPolya} uniformly over $\alpha \in [0, 1)$. We achieve this in the following theorem, which is the key ingredient in the proofs of Theorems \ref{Main} and \ref{MainPositive}. 
\begin{thm}\label{KeyResult}
Let $h(n)$ be a completely multiplicative function such that $|h(n)|\leq 1$ for all $n$. Let $x$ be large and put $Z=x^{21/40}$. There exists a constant $c>0$ such that for all real numbers $2\leq y\leq c\log x\log_4 x/(\log_3x)$  and $1/\log y\leq A\leq 4$,  the number of fundamental discriminants  $|d|\leq x$ such that
$$ \max_{\alpha\in [0, 1)}\Bigg|\sum_{\substack{1\leq n\leq Z\\ P^{+}(n)> y}} \frac{\chi_d(n) h(n) e(n\alpha)}{n}\Bigg|> e^{\gamma} A$$
is 
$$ \ll x\exp\left(-\frac{A^2y}{2 \log y} \left(1+O\left(\frac{\log_2 y}{\log y}+ \frac{\log _4x}{A\log_2x\log_3x}\right)\right)\right).$$

\end{thm}

Bober, Goldmakher, Granville and Koukoulopoulos \cite{BGGK} proved a similar result in the case of non-principal characters modulo a large prime $q$. However our proof differs from theirs as we shall now explain. Let $z=q^{21/40}$. In order to prove the analogue of Theorem \ref{KeyResult} for the family of non-principal characters modulo $q$, the authors of \cite{BGGK} established non-trivial upper bounds for the $2k$-th moment of 
\begin{equation}\label{MainTail}\max_{\alpha\in [0, 1)} \Bigg|\sum_{\substack{1\leq n\leq z\\ P^{+}(n)>y}}\frac{\chi(n)e(n\alpha)}{n}\Bigg|
\end{equation}
for $k$ roughly up to $\log q$. By decoupling the $\alpha$ from the character $\chi$ and expanding the $2k$-th moment of the inner sum in \eqref{MainTail}, one needs to control terms which have size as large as $z^k>q^{k/2}$ with $k$ up to $\log q$. This was possible in the case of the family of non-principal characters modulo $q$ thanks to the orthogonality relations of characters, which imply that the off-diagonal terms in these moments are given by $m\equiv n \mod q$ with $m\neq n$ and $m, n\leq z^k$. The authors of \cite{BGGK} proceed to bound these off-diagonal terms using intricate estimates involving the $k$-th divisor function. 

This argument is no longer valid for families of quadratic characters, since in this case only ``quasi-orthogonality relations'' are known (see for example Lemma 4.1 of \cite{GrSo03}), which allow one to control terms up to size $x^c$ for some $c<1$, if we run over fundamental discriminants $|d|\leq x$. To overcome this problem, we made the key observation %The first is that one does not need the condition $P^+(n)>y$ unless $n$ is ``small''. The second, which is the crucial observation, is 
that when $N<n<2N$ and $N$ is very large, one only needs to compute a small moment over the corresponding sum over the interval $[N, 2N]$, in order to show that this sum is small for most characters. More specifically, our approach consists of first splitting the sum in \eqref{MainTail} (in the case of quadratic characters associated to fundamental discriminants $|d|\leq x$) into two parts: the first over $n\leq Y$ and the second over $Y< n\leq Z$, where  $Y=(\log x)^{W(x)}$ is relatively ``small'' (our method allows one to choose $W(x)=\log_3 x/(\log_4 x)$). We first bound the $2k$-th moments of the first sum over $n\leq Y$ and show that only the diagonal terms contribute\footnote{We prefer to do this using an easy version of the quadratic large sieve (see \eqref{BabyLargeSieve} below), rather than employing the quasi-orthogonality relations,  in order to get a clean argument.} to these moments if $k\leq (\log x)/(3 \log Y)$ say. We then proceed to show that the second part over the large $n$'s is very small for most characters. To this end, we split it in dyadic intervals $N\leq n\leq 2N$ and then bound the 2$\ell_N$ moment of 
$$
\max_{\alpha\in [0, 1)} \Bigg|\sum_{N<n<2N}\frac{\chi(n)e(n\alpha)}{n}\Bigg|
$$
where $\ell_N$ is chosen such that $N^{\ell_N}$ is roughly of size $x$. This allows us to show that each of these parts is small for most characters using \emph{tailored} moments according to the size of $N$. In the range $Y<N<x^{\ep}$ we use \eqref{BabyLargeSieve} below, which is an easy version of the quadratic large sieve first due to  Elliott, while in the remaining range $x^{\ep}<N<Z$ we use the quadratic large sieve of Heath-Brown (see \eqref{HBLargeSieve} below). 

%We now deduce the upper bound of Theorem \ref{Main} from Theorem \ref{KeyResult}. 
\begin{proof}[Deducing the upper bound of Theorem \ref{Main} from Theorem \ref{KeyResult}]
Let $Z=x^{21/40}$ and $\delta=1/100$. By \eqref{Polya} we have 
\begin{equation}\label{SplitPolyaSmoothUnsmooth}
m(\chi_d) \leq \frac{e^{-\gamma}}{2} \max_{\alpha \in [0, 1)} \Bigg|\sum_{\substack{1\leq |n|\leq Z\\ P^{+}(n)\leq y}} \frac{\chi_d(n) (1-e(\alpha n))}{n}\Bigg| + 2e^{-\gamma} \max_{\alpha\in [0, 1)}\Bigg|\sum_{\substack{1\leq n\leq Z\\ P^{+}(n)> y}} \frac{\chi_d(n) e(n\alpha)}{n}\Bigg|+ O\left(x^{-\delta}\right),
\end{equation}
for any fundamental discriminant $-x\leq d\leq -x^{1-\delta}$. 
Since we do not have control over the first part over $y$-friable integers, we shall bound it using Corollary 3.5 of \cite{BGGK} which gives
\begin{equation}\label{CorBGGK} 
\max_{\alpha \in [0, 1)} \Bigg|\sum_{\substack{1\leq |n|\leq Z\\ P^{+}(n)\leq y}} \frac{\chi_d(n) (1-e(\alpha n))}{n}\Bigg| \leq 
2e^{\gamma} \log y+ 2\log 2 + O\left(\frac{\log\log y}{\log y}\right).
\end{equation}
We now set $y= e^{\tau- e^{-\gamma} \log 2 - 2 B}$ where $B >0 $ is a parameter to be chosen. Combining this last estimate with \eqref{SplitPolyaSmoothUnsmooth} and \eqref{CorBGGK} we deduce that 
$$ 
m(\chi_d)  \leq  \tau -2B+ 2 e^{-\gamma} \max_{\alpha \in [0, 1)} \Bigg|\sum_{\substack{1\leq n\leq Z \\ P^{+}(n)> y}} \frac{\chi_d(n) e(\alpha n)}{n}\Bigg| +  O\left(\frac{\log \tau }{\tau}\right).
$$ 
Thus, the proportion of fundamental discriminants $-x\leq d\leq -x^{1-\delta}$ such that $m(\chi_d)>\tau$ is bounded by the proportion of fundamental discriminants $|d|\leq x$ such that 
$$ 
\max_{\alpha \in [0, 1)} \Bigg|\sum_{\substack{1\leq n\leq Z\\ P^{+}(n)> y}} \frac{\chi_d(n) e(\alpha n)}{n}\Bigg|>e^{\gamma} \left(B- C_0 \frac{\log \tau}{\tau}\right),
$$
for some suitably large constant $C_0>0$.
Choosing $B=1$ and appealing to Theorem \ref{KeyResult} completes the proof.

\end{proof}

We now turn our attention to the lower bound of Theorem \ref{Main} which is easier. For $d<0$ one has the identity (see Theorem 9.21 of \cite{MVBook})
$$ \sum_{n\leq |d|/2} \chi_d(n)= (2-\chi(2))\frac{\mathcal{G}(\chi_d)}{i\pi} L(1, \chi_d),$$
which implies
\begin{equation}\label{LowerBoundML}
m(\chi_d) \geq e^{-\gamma}L(1, \chi_d).
\end{equation}

Therefore, one can immediately deduce a corresponding lower bound for the distribution function $\Psi_x^-(\tau)$ from the following result, which follows from a straightforward adaptation of the work of Granville and Soundararajan \cite{GrSo03} on the distribution of $L(1, \chi_d)$.
\begin{thm}[Granville-Soundararajan]\label{GrSoThm}
Let $B_0$ be the constant in Theorem \ref{Main}, $\psi$ be a character modulo some $b\in \{1,3\}$, and $1\leq \tau \leq \log_2x+\log_4 x-20-M$ for some $M\geq 0$. Then we have 
$$\frac{1}{|\mathcal{F}^{\pm}(x)|}\left|\left\{d\in \mathcal{F}^{\pm}(x) : L(1, \chi_d \psi)>\frac{\phi(b)}{b} e^{\gamma} \tau\right\}\right|= \exp\left(-\frac{e^{\tau-B_0}}{\tau}\left(1+O\left(\frac{1}{\tau}+ e^{-e^M}\right)\right)\right).$$
\end{thm}
One can do a little better by combining our Theorem \ref{KeyResult} with the following result, which we extract from the proof of Theorem 1.2 of  \cite{BGGK}.  

\begin{thm} [Bober-Goldmakher-Granville-Koukoulopoulos]\label{BGGKOdd}
Let $C$ be a positive constant. Let $q$ be a large positive integer, and $y\leq (\log q)^2$ be a large real number. Let $\chi$ be a primitive odd character modulo $q$ such that 
$$|L(1, \chi)| > e^{\gamma} \log y- C \quad \textup{ and  } \quad \max_{\alpha\in [0, 1)} \Bigg|\sum_{\substack{n\geq 1\\P^{+}(n)> y}} \frac{\chi(n)e(n\alpha)}{n}\Bigg|\leq 1.$$
Then we have
$$
m(\chi) > e^{-\gamma} \big(|L(1, \chi)|+\log 2\big)+O\left(\frac{(\log\log  y)^2}{\sqrt{\log y}}\right).
$$ 

\end{thm}

\begin{rem}
We note that $q$ is assumed to be prime in the statement of Theorem 1.2 of \cite{BGGK}. However,  this is only used to show that many non-principal characters $\chi \bmod q$ satisfying the assumptions of Theorem \ref{BGGKOdd} exist. Indeed, the proof of Theorem \ref{BGGKOdd} (see the end of Section 4 of \cite{BGGK}) only uses estimates on exponential sums over $y$-friable integers (Lemmas 3.2 and 3.4 of \cite{BGGK})  together with ideas of Bober \cite{Bob} on averages of character sums to arbitrary moduli.   
\end{rem}

\begin{proof}[Deducing the lower bound of Theorem \ref{Main} from Theorems \ref{KeyResult}, \ref{GrSoThm} and \ref{BGGKOdd}]
Let $Z=x^{21/40}$ and $\delta=1/100$. By partial summation and the P\'olya-Vinogradov inequality, it follows that 
\begin{equation}\label{PolyaUniformTail}
\max_{\alpha\in [0, 1)} \Bigg|\sum_{n>Z}\frac{\chi_d(n)e(n\alpha)}{n}\Bigg| \ll \frac{\sqrt{|d|}\log |d|}{Z} \ll x^{-\delta},
\end{equation}
for any fundamental discriminant $|d|\leq x$. Furthermore, by Lemma 3.2 of \cite{BGGK} we have 
\begin{equation}\label{PolyaUniformTail3}
\max_{\alpha\in [0, 1)} \Bigg|\sum_{\substack{n>Z \\ P^{+}(n)\leq y}}\frac{\chi_d(n)e(n\alpha)}{n}\Bigg| \leq \sum_{\substack{n>Z \\ P^{+}(n)\leq y}} \frac{1}{n} \ll e^{-\sqrt{\log y}}, 
\end{equation}
for any real number $2\leq y\leq (\log x)^2$. Combining these estimates implies that 
\begin{equation}\label{PolyaUniformTail2}
\max_{\alpha\in [0, 1)} \Bigg|\sum_{\substack{n>Z \\ P^{+}(n)> y}}\frac{\chi_d(n)e(n\alpha)}{n}\Bigg|  \ll e^{-\sqrt{\log y}},
\end{equation}
for any fundamental discriminant $|d|\leq x$ and  any real number $2\leq y\leq (\log x)^2$.

Let $C_1, C_2>0$ be suitably large constants and put $y=e^{\tau+ C_1}$. If $C_1$ is large enough, then combining  Theorems \ref{KeyResult} and \ref{GrSoThm} we deduce that the number of fundamental discriminants $-x<d<0$ such that 
$$L(1, \chi_d)>e^{\gamma} \tau-\log 2 + C_1\frac{(\log \tau)^2}{\sqrt{\tau}}$$
and 
$$ \max_{\alpha\in [0, 1)} \Bigg|\sum_{\substack{1\leq n\leq Z\\P^{+}(n)> y}} \frac{\chi_d(n)e(n\alpha)}{n}\Bigg|\leq 1-\frac{C_2}{\tau}$$
is 
$$ \geq \exp\left(-\frac{e^{\tau-\eta
-B_0}}{ \tau} \left(1+O\left(\frac{(\log \tau)^2 }{\sqrt{\tau}}\right)\right)\right). $$
By \eqref{PolyaUniformTail2} we deduce that for any such discriminant $d$ we have 
$$ \max_{\alpha\in [0, 1)} \Bigg|\sum_{\substack{n\geq 1\\P^{+}(n)> y}} \frac{\chi_d(n)e(n\alpha)}{n}\Bigg| \leq 1,$$ and hence it follows from Theorem \ref{BGGKOdd} that  $m(\chi_d)>\tau$ if $C_1$ is suitably large. This completes the proof.
\end{proof}

\subsection{The case of positive discriminants: Proof of Theorem \ref{MainPositive}}
Recall that $\chi_d$ is an even character if $d$ is a positive fundamental discriminant. In this case, we shall appeal to the following structure result for even characters with large sums, which we extract from the proof of Theorem 2.3 of  \cite{BGGK}.

\begin{thm}[Bober-Goldmakher-Granville-Koukoulopoulos]\label{BGGKEven}
Let $C$ be a positive constant. Let $q$ be a large positive integer and $y\leq (\log q)^2$ be a large real number.  Let $\psi\mod q$ be a primitive even character such that 
$$ \max_{\alpha\in [0, 1)} \Bigg|\sum_{\substack{n\geq 1\\ P^{+}(n)\leq y}} \frac{\psi(n)e(n\alpha)}{n}\Bigg|> \frac{e^{\gamma}}{\sqrt{3}}\log y-C\log\log y, \  \textup{ and  } \ 
\max_{\alpha\in [0, 1)} \Bigg|\sum_{\substack{n\geq 1\\P^{+}(n)> y}} \frac{\psi(n)e(n\alpha)}{n}\Bigg|\leq 1.
$$
Then we have 
$$m(\psi)= \frac{e^{-\gamma}\sqrt{3}}{2} \big|L(1, \psi \chi_{-3})\big| + O(\log \log y).$$
\end{thm}

\begin{rem}
Here again we note that $q$ is assumed to be prime in Theorem 2.3 of \cite{BGGK},  but this is only used to show that many non-principal characters $\psi \bmod q$ satisfying the assumptions of Theorem \ref{BGGKEven} exist. Indeed, all of the ingredients of the proof of Theorem \ref{BGGKEven} hold for an arbitrary primitive even character $\psi$ (see Section 8 of \cite{BGGK}), and are derived using the ``pretentious'' theory of character sums developed by Granville and Soundararajan in \cite{GrSo07}, along with estimates of Montgomery and Vaughan \cite{MV77} on exponential sums with multiplicative coefficients.  
\end{rem}

%\begin{rem}
%We observe that 
%$$ 
%\max_{\alpha\in [0, 1)} \left|\sum_{\substack{n\in \mathbb{Z}\setminus{0}\\P^{+}(n)\leq y}} \frac{\psi(n)e(n\alpha)}{n}\right| \leq 2\sum_{\substack{n\geq 1\\ P^{+}(n)\leq y}}\frac{1}{n} = 2 e^{\gamma} \log y +O(1), 
%$$
%by Mertens Theorem. Therefore, Theorem \ref{BGGKEven} implies that if the maximum of the truncation of the sum of $\psi(n)e(n\alpha)/n$ over $y$-friable integers is close to be maximal, and the tail of the sum is small, then $\psi$ ``pretends'' to be the character $\chi_{-3}$ and $m(\psi)$ is (up to a constant factor) roughly the same as $|L(1, \psi\chi_{-3})|$. 
%\end{rem}

\begin{proof}[Deducing Theorem \ref{MainPositive} from Theorems \ref{KeyResult}, \ref{GrSoThm} and \ref{BGGKEven}]
The lower bound follows immediately from Theorem \ref{GrSoThm} together with the standard inequality (see for example the beginning of Section 4 of \cite{BGGK})
\begin{equation}\label{LowerBoundMLEven}
M(\psi) \geq \frac{\sqrt{3q}}{2\pi} |L(1, \psi\chi_{-3})|,
\end{equation}
which is valid for any primitive even character $\psi$ modulo $q$.

We now prove the upper bound. First, it follows from \eqref{PolyaFourier} that for all fundamental discriminants $0<d<x$ we have
\begin{equation}\label{PolyaEVEN}
m(\chi_d)=  e^{-\gamma}\max_{\alpha \in [0, 1)} \Bigg|\sum_{1\leq n\leq Z} \frac{\chi_d(n) \sin(2\pi n\alpha)}{n}\Bigg|+O(1), 
\end{equation}
since $\chi_d$ is even. 
Let $C_1, C_2>0$ be suitably large constants,  $y=e^{\sqrt{3}\tau+C_1}$, and define $\mathcal{D}^+(x, y)$ to be the set of fundamental discriminants $0<d\leq x$ such that 
$$ m(\chi_d)> \tau \quad \textup{ and } \max_{\alpha\in [0, 1)} \Bigg|\sum_{\substack{1\leq n\leq Z\\P^{+}(n)> y}} \frac{\chi_d(n)e(n\alpha)}{n}\Bigg|\leq 1 -\frac{C_2}{\tau}.$$
By combining the lower bound of Theorem \ref{MainPositive} with Theorem \ref{KeyResult} we deduce that if $C_1$ is suitably large then 
\begin{equation}\label{SubsetPosDiscrim}
\frac{|\mathcal{D}^+(x,y)|}{|\mathcal{F}^+(x)|} = \Psi_x^+(\tau)(1+o(1)).
\end{equation}
Now, let $d\in \mathcal{D}^+(x, y)$. By \eqref{PolyaUniformTail3} and \eqref{PolyaEVEN} together with our assumption on $d$ we have 
\begin{align*}
m(\chi_d)&=  e^{-\gamma}\max_{\alpha \in [0, 1)} \Bigg|\sum_{\substack{1\leq n\leq Z\\P^{+}(n)\leq y }} \frac{\chi_d(n) \sin(2\pi n\alpha)}{n}\Bigg|+O(1)\\
&= e^{-\gamma}\max_{\alpha \in [0, 1)} \Bigg|\sum_{\substack{n\geq 1\\P^{+}(n)\leq y }} \frac{\chi_d(n) \sin(2\pi n\alpha)}{n}\Bigg|+O(1).
\end{align*}
Since $m(\chi_d)>\tau$ we deduce that
$$ \max_{\alpha\in [0, 1)}\Bigg|\sum_{\substack{n\geq 1\\P^{+}(n)\leq y }} \frac{\chi_d(n)e(\alpha n)}{n}\Bigg|\geq e^{\gamma} m(\chi_d)+O(1) > \frac{e^{\gamma}}{\sqrt{3}} \log y -C_3, $$
for some positive constant $C_3$. Moreover, by \eqref{PolyaUniformTail2} and our assumption on $d$ we get 
$$\max_{\alpha\in [0, 1)} \Bigg|\sum_{\substack{n\geq 1\\P^{+}(n)> y}} \frac{\chi_d(n)e(n\alpha)}{n}\Bigg|\leq 1-\frac{C_2}{\tau}+O\left(e^{-\sqrt{\tau}}\right)\leq 1, $$
if $C_2$ is suitably large.
Thus $\chi_d$ satisfies the conditions of Theorem \ref{BGGKEven}, which implies that 
$$ m(\chi_d)= \frac{e^{-\gamma}\sqrt{3}}{2}|L(1, \chi_d\chi_{-3})|+O(\log \tau).$$
Therefore, for all $d\in \mathcal{D}^+(x,y)$ we have 
$$ |L(1, \chi_d\chi_{-3})| > \frac{2e^{\gamma}}{\sqrt{3}} \tau - C_4 \log \tau, $$
for some constant $C_4>0$. Hence, it follows from Theorem \ref{GrSoThm} that 
$$ |\mathcal{D}^+(x,y)| \ll |\mathcal{F}^+(x)|\exp\left(-\frac{e^{\sqrt{3}\tau}}{\tau^{C_5}}\right), $$
for some constant $C_5>0$. Combining this estimate with \eqref{SubsetPosDiscrim} completes the proof.
\end{proof}

%%%%%%%%%%%%%%%%%%%%%%%%%%%%%%%%%%%%%%%%%%%%%%%%%%%%%%%%%%%%%%%%

\subsection{The family of Legendre symbols: Outline of the proofs of Theorems \ref{MainPrimes} and \ref{MainPositivePrimes}}
In order to prove the upper and lower bounds of Theorems \ref{MainPrimes} and \ref{MainPositivePrimes}, we shall follow the same strategy of the proofs of Theorems \ref{Main} and \ref{MainPositive}. To this end we establish the following key result, which is the analogue of Theorem \ref{KeyResult} for prime discriminants. To shorten our notation, we let $\psi_p$ denote the Legendre symbol modulo $p$ throughout the remaining part of the paper. 
\begin{thm}\label{KeyResultPrimes}
Let $h(n)$ be a completely multiplicative function such that $|h(n)|\leq 1$ for all $n$. Let $x$ be large and put $Z=x^{21/40}$. There exists a constant $c>0$ such that for all real numbers $2\leq y\leq c\log x\log_4 x/(\log_3x)$  and $1/\log y\leq A\leq 4$,  the number of primes   $p\leq x$ such that
$$ \max_{\alpha\in [0, 1)}\Bigg|\sum_{\substack{1\leq n\leq Z\\ P^{+}(n)> y}} \frac{\psi_p(n)h(n) e(n\alpha)}{n}\Bigg|> e^{\gamma} A$$
is 
$$ \ll \pi(x) \exp\left(-\frac{A^2y}{2 \log y} \left(1+O\left(\frac{\log_2 y}{\log y}+ \frac{\log _4x}{A\log_2x\log_3x}\right)\right)\right).$$

\end{thm}
We end this section by deducing Theorems \ref{MainPrimes} and \ref{MainPositivePrimes} from this result along with Theorem \ref{LamzouriL1chi} on the distribution of large values of $L(1, \psi_p)$. 
\begin{proof}[Proof of Theorem \ref{MainPrimes} assuming Theorems \ref{LamzouriL1chi} and \ref{KeyResultPrimes}]
First, Part 1 of Theorem \ref{MainPrimes} can be derived along the same lines of the proof of the upper bound of Theorem \ref{Main} by replacing Theorem \ref{KeyResult} by Theorem \ref{KeyResultPrimes}. 

Now, the proof of Part 2 follows along the same lines of the proof of the lower bound of Theorem \ref{Main} by replacing Theorems \ref{KeyResult} and \ref{GrSoThm} by Theorem \ref{KeyResultPrimes} and the estimate \eqref{EstimateDistribL1psi} respectively. Finally, the proof of Part 3 follows from \eqref{EstimateDistribL1psi2} together with the lower bound \eqref{LowerBoundML}.
\end{proof}
\begin{proof}[Proof of Theorem \ref{MainPositivePrimes} assuming Theorems \ref{LamzouriL1chi} and \ref{KeyResultPrimes}]
Part 2 and the lower bound of Part 1 of Theorem \ref{MainPositivePrimes} follow immediately from Theorem \ref{LamzouriL1chi} together with the lower bound \eqref{LowerBoundMLEven}. 
Finally, the upper bound of Part 1 can be derived along the same lines of the proof of the upper bound of Theorem \ref{MainPositive} with the choice $y=\tau^2 e^{\sqrt{3}\tau+C_1}$ for some suitably large constant $C_1$, by using Theorem \ref{BGGKEven} and replacing Theorems \ref{KeyResult} and \ref{GrSoThm}  by Theorems \ref{KeyResultPrimes} and \ref{LamzouriL1chi} respectively.  

\end{proof}
%%%%%%%%%%%%%%%%%%%%%%%%%%%%%%%%%%%%%%%%%%%%%%%%%%%%
\subsection{The plan of the paper} The plan of the paper is as follows. In Section 3 we gather several preliminary results on sums of divisor functions and moments of Random multiplicative functions, which will shall use throughout the paper. Section 4 will be devoted to the proof of Theorem \ref{KeyResult}. Theorem \ref{KeyResultPrimes} will be established in Section 5. In Section 6 we investigate the distribution of large values of $L(1, \psi_p)$ and prove Theorem \ref{LamzouriL1chi}. Finally, in Section 7, we investigate the positivity of sums of Legendre symbols and prove Theorems \ref{PositivityLegendreTheorem} and \ref{NegativityLegendreTheorem}.

%%%%%%%%%%%%%%%%%%%%%%%%%%%%%%%%%%%%%%%%%%%%%%%%%%%%%%%%%%%%%%%%%%%%%%%%%%%%%%%%%%%%%%%%%%%%%%%%%%%%%%%%%%%%%%%%%%%%%%%%%%%%%

\section{Sums of divisor functions and random multiplicative functions}
In this section we gather together several preliminary results concerning sums of divisor functions, which are related to certain moments over random multiplicative functions.  We let $\{\X(n)\}_{n\ge 1}$ be Rademacher random multiplicative functions, that is $\{\X(p)\}_{p \text{ prime}}$ are I.I.D. random variables taking the values $\pm 1$ with equal probability $1/2$, and we extend $\X(n)$ multiplicatively to all positive integers by setting $\X(1)=1$ and $\X(n)= \prod_{p^{\ell} || n} \X(p)^{\ell}$.  We start with the following lemma. 

%\begin{lem}[Lemma 2.1 of \cite{LaActa}] \label{SumPrimes} 
%Let $x$ be large and  $0\leq \alpha\leq\log_3x/(4\log x)$. Then for $1\leq A<4$ we have
%$$\sum_{p\leq x^{A}}\frac{1}{p^{1-\alpha}}=(1+o(1))\log\log x .$$
%\end{lem}

\begin{lem}\label{DivisorSums} Let $k$ be a large real number. Then for any $0\leq \alpha \leq \log_3 k/(2\log k)$ we have
\begin{equation}\label{Divisor3}
\sum_{n=1}^{\infty}\frac{d_k(n)^2}{n^{2-\alpha}}=\exp\big(O(k\log\log k)\big),
\end{equation}
and 
\begin{equation}\label{Divisor1}
\sum_{n=1}^{\infty}\frac{d_k(n^2)}{n^{2-\alpha}}=\exp\big(O(k\log\log k)\big).
\end{equation}
Moreover, for all $y>k$ we have 
\begin{equation}\label{Divisor2}
\sum_{P^{-}(n)>y}\frac{d_k(n^2)}{n^{2}}= \exp\left(O\left(\frac{k^2}{y\log y}\right)\right),
\end{equation}
where here and throughout $P^{-}(n)$ denotes the smallest prime factor of $n$.
\end{lem}

\begin{proof} The bound \eqref{Divisor3} follows from Lemma 3.3 of \cite{La11}.  Furthermore, the estimate \eqref{Divisor1} follows from \eqref{Divisor3} upon noting  that $d_k(n^2)\leq d_k(n)^2.$
%We let $\{\X(p)\}_{p}$ be independent random variables, taking the values $\pm 1$ with equal probability $1/2$. We then expand $\X(n)$ manipulatively to all positive integers by setting $X(n)=1$ and $\X(n)= \prod_{p^{\ell} || n} \X(p)^{\ell}$.  
We now establish \eqref{Divisor2}. Let $p>y$ be a prime number. Then we have 
\begin{equation}\label{MRandomLPrimes}
\begin{aligned}
\ex\left(\left(1-\frac{\X(p)}{p}\right)^{-k}\right)
&= \frac{1}{2}\left(\left(1-\frac{1}{p}\right)^{-k}+ \left(1+\frac{1}{p}\right)^{-k}\right)\\
& = \frac{1}{2}\left(\exp\left(\frac{k}{p}+ O\left(\frac{k}{p^{2}}\right)\right)+ \exp\left(-\frac{k}{p}+ O\left(\frac{k}{p^{2}}\right)\right)\right)\\
& = 1+O\left(\frac{k^2}{p^{2}}\right).
\end{aligned}
\end{equation} Hence we derive
\begin{align*}
\sum_{P^{-}(n)>y}\frac{d_k(n^2)}{n^{2}} & = \ex\left(\sum_{P^{-}(n)>y}\frac{d_k(n)\X(n)}{n}\right)= \prod_{p>y} \ex\left(\left(1-\frac{\X(p)}{p}\right)^{-k}\right)\\&= \exp\left(O\left(k^2\sum_{p>y} \frac{1}{p^2}\right)\right)= \exp\left(O\left(\frac{k^2}{y\log y}\right)\right),
\end{align*}
as desired. 

\end{proof}

In order to prove Theorems \ref{KeyResult} and \ref{KeyResultPrimes} we need a uniform bound for the moments
\begin{equation}\label{DefinitionMyk} 
\mathcal{M}_{y}(k):=\ex\Bigg(\Bigg|\sum_{\substack{n>1\\ P^{-}(n)>y}} \frac{\X(n)}{n}\Bigg|^{2k}\Bigg).  
\end{equation}
To this end we establish the following result. 
\begin{pro}\label{MomentsUnsmooth}
Let $k\ge 2$ be a positive integer and $(k\log k)/10 <y$ be real numbers. Then we have 
$$
\mathcal{M}_{y}(k) \leq e^{O(k\log\log y/\log y)} \left(\frac{2k}{ey\log y}\right)^k.
$$
\end{pro}
\begin{rem}
A similar bound was obtained by Bober, Goldmakher, Granville and Koukoulopoulos in the case of Steinhaus Random multiplicative functions (see the end of Section 5 of \cite{BGGK}). However, our argument is easier, and can be adapted to recover this case as well.
\end{rem}
\begin{proof}[Proof of Proposition \ref{MomentsUnsmooth}]

By expanding the moment and using that $\ex(\X(n))=1$ if $n$ is a square and equals $0$ otherwise, we get 
\begin{equation}\label{MykExpression}
 \mathcal{M}_{y}(k)= \ex\left(\sum_{P^{-}(n)>y}\frac{\widetilde{d_{2k}}(n)}{n} \X(n)\right)= \sum_{P^{-}(n)>y}\frac{\widetilde{d_{2k}}(n^2)}{n^2},
\end{equation}
where 
$$ \widetilde{d_k}(n):= \left|\left\{(n_1, \dots, n_k) \in \mathbb{N}^k, \textup{ such that } n_j>1 \textup{ for all } j, \text{ and } n_1\cdots n_k= n\right\}\right|.$$
Therefore, it follows from \eqref{Divisor2} that 
\begin{equation}\label{FirstBoundM}
 \mathcal{M}_{y}(k)\leq \sum_{P^{-}(n)>y}\frac{d_{2k}(n^2)}{n^{2}} =\exp\left(O\left(\frac{k^2}{y\log y}\right)\right).
\end{equation}
To obtain a better bound for  $\mathcal{M}_{y}(k)$ we consider the following ``good'' event 
$$\mathcal{A} :=\{ |\Y| \leq  \varepsilon\}, \text{ where } \Y:= \sum_{p>y} \frac{\X(p)}{p}, $$
and $0<\varepsilon<1$ is a small parameter to be chosen. Note that 
$$
\pr(\mathcal{A}^c)\leq \ep^{-2\ell}\ex(|\Y|^{2\ell})\ll  \left(\frac{3\ell}{e \ep^2 y\log y}\right)^{\ell},
$$
since 
\begin{equation}\label{MomentsSumPrime}
\ex(|\Y|^{2\ell})= \sum_{\substack{p_1, \dots, p_{2\ell}>y \\ p_1\cdots p_{2\ell}=\square}}\frac{1}{p_1\dots p_{2\ell}} \leq \frac{(2\ell)!}{2^{\ell} \ell!}\left(\sum_{p>y} \frac{1}{p^2}\right)^{\ell}\ll e^{O(\ell/\log y)} \left(\frac{2\ell}{ey\log y}\right)^{\ell},  
\end{equation} 
which follows from the bounds $(2\ell)!/(2^{\ell} \ell!)\ll (2\ell/e)^{\ell}$ by Stirling's formula, and $\sum_{p>y} 1/p^2= 1/(y\log y)+ O\left(1/(y(\log y)^2)\right)$ by the Prime Number Theorem. 
Choosing $\ell=\lfloor (\ep^2 y\log y)/3 \rfloor$ gives 
\begin{equation}\label{BoundProbComplement}
\pr(\mathcal{A}^c) \ll \exp\left(-\frac{\ep^2 y\log y}{3}\right).
\end{equation}
We now split the moment $\mathcal{M}_{y}(k)$ into two parts 
\begin{equation}\label{SplitM}
\mathcal{M}_{y}(k)= \ex\Bigg(\Bigg|\sum_{\substack{n>1\\ P^{-}(n)>y}}  \frac{\X(n)}{n}\Bigg|^{2k} \cdot \mathbf{1}_{\mathcal{A}}\Bigg)+\ex\Bigg(\Bigg|\sum_{\substack{n>1\\ P^{-}(n)>y}}  \frac{\X(n)}{n}\Bigg|^{2k}\cdot \mathbf{1}_{\mathcal{A}^c}\Bigg),
\end{equation}
where $\mathbf{1}_{\mathcal{B}}$ is the indicator function of the event $\mathcal{B}$. By the Cauchy-Schwarz inequality and the estimates \eqref{FirstBoundM} and \eqref{BoundProbComplement}, the contribution of the second part is 
\begin{equation}\label{ContrSecondPart}\ex\Bigg(\Bigg|\sum_{\substack{n>1\\ P^{-}(n)>y}} \frac{\X(n)}{n}\Bigg|^{2k}\cdot \mathbf{1}_{\mathcal{A}^c}\Bigg) \leq  \mathcal{M}_{y}(2k)^{1/2} \cdot \pr(\mathcal{A}^c)^{1/2} \ll \exp\left(-\frac{\ep^2 y\log y}{6}+ O\left(\frac{k^2}{y\log y}\right)\right).
\end{equation}
We now estimate the contribution of the first part in \eqref{SplitM}. Note that on the event $\mathcal{A}$ we have  $|e^{\Y}-1|\leq e^{\varepsilon} |\Y|$ and 
$$ 
\sum_{\substack{n>1\\ P^{-}(n)>y}}  \frac{\X(n)}{n} = -1+\prod_{p>y}\left(1-\frac{\X(p)}{p}\right)^{-1}= -1+ e^{\Y +O(1/y\log y)}= e^{\Y}-1+ O\left(\frac{1}{y\log y}\right).
$$
Therefore, using Minkowski's inequality and \eqref{MomentsSumPrime} we derive 
$$\ex\Bigg(\Bigg|\sum_{\substack{n>1\\ P^{-}(n)>y}}  \frac{\X(n)}{n}\Bigg|^{2k} \cdot \mathbf{1}_{\mathcal{A}}\Bigg)^{1/2k}\le e^{\ep} \ex(|\Y|^{2k})^{1/2k}+ O\left(\frac{1}{y\log y}\right)\leq e^{\ep+O(1/\log y)} \sqrt{\frac{2k}{ey\log y}}.$$
Choosing $\ep= (\log\log y) /\log y$ and combining this estimate with \eqref{ContrSecondPart} completes the proof. 

%We now use the following inequality from \cite{BGGK} (see  \cite[eq. (5.5)]{BGGK}) 
%$$ 
%(a+b)^{2k} \leq (a^{2k}+b^k) e^{\sqrt{b}/2k},
%$$ 
%valid for all $a, b>0$ and positive integers $k$. Therefore, there exists a positive constant $c>0$ such that 
%\begin{align*}
%\ex\Bigg(\Bigg|\sum_{\substack{n\in P(y, z)\\ n>1}} \frac{\X(n)}{n}\Bigg|^{2k} \cdot \mathbf{1}_{\mathcal{A}}\Bigg)
%\leq \ex\left(\left(e^{\ep} |\Y|+ \frac{c}{y\log y}\right)^{2k}\right) \leq e^{\ep k+ O(1/(k\sqrt{y}\log y))} \ex\left(|\Y|^{2k}\right) +\left(\frac{c}{y\log y}\right)^k
%\end{align*}
\end{proof}

%We shall now bound bound the proportion of $d\in \F$ such that the tail of the sum over $n$ is large. To this end we first record the following two large sieve inequalities for quadratic characters, 

%%%%%%%%%%%%%%%%%%%%%%%%%%%%%%%%%%%%%%%%%%%%%%%%%%%%%%%%%%%%%%%%%%%%%%%%%%%%%%%%%%%%%%%%%%%%%%%%%%%%%%%%%%%%%%%%%%%%%%%%%%%%%%%%%%%%%%%%%%%%%%%%%%%%%%%%%%%%%%%%%%%%%%%%%%%%%%%%%%%%%%%%%%%%%%%%%%%%%%%%%%%%%%%%%%%%%%%%%%%%%%%%%%%%%%%%%%%%%%%%%%%%%%%%%%%%%%%%%%%%%%%%%%%%%%%%%%%%%%%%
\section{The distribution of the tail in P\'olya's Fourier expansion: Proof of Theorem \ref{KeyResult}}

%We now proceed to bound the probability that the tail is large. To this end we shall split the sum into four parts... We shall use \eqref{BabyLargeSieve} when $Y\leq \sqrt{x}$, and \eqref{HBLargeSieve} when $Y$ is of the same size as $x$. 
%In order to bound the tail of the sum over $n$ in \eqref{Polya}, we shall first split it in dyadic intervals. 
We start by recording two large sieve inequalities for quadratic characters, the most important of which is due to Heath-Brown \cite{HB}.
\begin{lem}\label{LargeSieve}
Let $x, N\geq 2$. Then for arbitrary complex numbers $a_n$ we have 
\begin{equation}\label{BabyLargeSieve}
\sum_{d\in \Fo} \left|\sum_{n\leq N} a_n \chi_d(n)\right|^2 \ll (x+ N^2 \log N) \sum_{\substack{m, n \leq N\\ mn= \square}} |a_ma_n|,
\end{equation}
and for any $\ep>0$ we have
\begin{equation}\label{HBLargeSieve}
\sum_{d\in \Fo} \left|\sum_{n\leq N} a_n \chi_d(n)\right|^2 \ll_{\ep} (xN)^{\ep}(x+N) \sum_{\substack{m, n \leq N\\ mn= \square}} |a_ma_n|.
\end{equation}
\end{lem}
\begin{proof} The first inequality is standard and is a straightforward application of the P\'olya-Vinogradov inequality. It can be found for example in Lemma 1 of \cite{BaMo}, and can be traced back to the work of Elliott \cite{El}. The second inequality, which is deeper, was established by Heath-Brown (see Corollary 2 of \cite{HB}).
\end{proof}
Let $Y:= \exp(C\log_2 x \log _3 x/\log_4 x)$, where $C>0$ is a suitably large constant, and $Y\leq N\leq x^{21/40}$ be a real number. Using Lemma \ref{LargeSieve}, we shall first prove that for all fundamental discriminants $|d|\leq x$ except for a small exceptional set $\mathcal{E}(x)$, the quantity 
\begin{equation}\label{KeyQuantity}
\max_{\alpha\in [0, 1)} \Bigg|\sum_{N\leq n\leq 2N}\frac{\chi_d(n)a_ne(n\alpha)}{n}\Bigg|
\end{equation}
is small, where $\{a_n\}_{n\geq 1}$ is an arbitrary sequence of complex numbers such that $|a_n|\leq 1$ for all $n$.  We shall use Heath-Brown's large sieve \eqref{HBLargeSieve} if $N$ is in the range $x^{\ep}\leq N\leq x^{21/40}$, and Elliott's large sieve \eqref{BabyLargeSieve} in the remaining range $Y \leq N\leq x^{\ep}$.
\begin{pro}\label{Tail1} Let $A, \ep>0$ be fixed and put $\delta=1/100$. Let $\{a_n\}_{n\geq 1}$ be an arbitrary sequence of complex numbers such that $|a_n|\leq 1$ for all $n$.  Let $N_1, N_2$ be real numbers such that $x^{\ep} \leq N_1 < N_2\leq 2N_1\leq  x^{21/40}$ be real numbers. Then there are at most $O_{\ep, A}(x^{1-\delta})$ fundamental discriminants $|d|\leq x$ such that 
$$
\max_{\alpha\in [0, 1)}\left|\sum_{N_1\leq n\leq N_2} \frac{\chi_d(n)a_n e(\alpha n)}{n}\right|\geq \frac{1}{(\log N_1)^A}.
$$
\end{pro}

\begin{pro}\label{Tail2} Let $A>0$ be a fixed constant. Let $\{a_n\}_{n\geq 1}$ be an arbitrary sequence of complex numbers such that $|a_n|\leq 1$ for all $n$. There exist positive constants $\ep$ and $C$ (which depend at most on $A$) such that for all real numbers $N_1, N_2$ verifying  

\noindent $ \exp(C\log_2 x \log _3 x/\log_4 x)\leq N_1 < N_2\leq 2N_1\leq x^{\ep}$, the number of fundamental discriminants $|d|\leq x$ such that 
$$ 
\max_{\alpha\in [0, 1)}\left|\sum_{N_1\leq n\leq N_2} \frac{\chi_d(n)a_n e(\alpha n)}{n}\right|\geq \frac{1}{(\log N_1)^A},
$$
is 
$$ \ll x\exp\left(-\frac{\log x \log_4 x}{10 \log_2 x}\right).$$
\end{pro}

In order to prove these results we shall bound suitable moments of \eqref{KeyQuantity}. We first start with the following easy lemma which reduces the problem of bounding these moments to a setting where we can apply the large sieve.
\begin{lem}\label{Reduction}
Let $\{a_n\}_{n\geq 1}$ be an arbitrary sequence of complex numbers such that $|a_n|\leq 1$ for all $n$. Let $\mathcal{D}$ be a set of Dirichlet characters, and $2\leq N_1 <N_2\leq R$ be real numbers. Define 
$\mathcal{A}= \{b/R : 1\leq b\leq R\}.$ Then for any positive integer $k\geq 1$ we have 
\begin{align*}
&\Bigg(\sum_{\chi\in \mathcal{D}} \max_{\alpha\in [0,1)} \Bigg|\sum_{N_1\leq  n\leq N_2} \frac{\chi(n) a_n e(\alpha n)}{n}\Bigg|^{2k}\Bigg)^{1/2k} \\
& \leq \Bigg(\sum_{\alpha \in \mathcal{A}}  \sum_{\chi\in \mathcal{D}} \Bigg|\sum_{N_1^k\leq n\leq N_2^k}\frac{\chi(n)g_{N_1, N_2, k}(n, \alpha)}{n}\Bigg|^{2}\Bigg)^{1/2k}+ O\left(|\mathcal{D}|^{1/2k} N_2R^{-1}\right),
\end{align*}
where 
\begin{equation}\label{Defgn}
g_{N_1, N_2, k}(n, \alpha):= \sum_{\substack{N_1\leq n_1, \dots, n_k \leq N_2\\ n_1\cdots n_k= n}}\prod_{j=1}^k a_{n_j}e(\alpha n_j),
\end{equation}
and the implicit constant in the error term is absolute. 
\end{lem}
\begin{proof}
First, we observe that for all $\alpha\in [0, 1)$ there exists $\beta_{\alpha}\in \mathcal{A}$ such that $|\alpha-\beta_{\alpha}|\leq 1/R.$
In this case we have $e(\alpha n)= e(\beta_{\alpha} n)+ O(n/R),$
 and hence 
$$\max_{\alpha\in [0,1)} \Bigg|\sum_{N_1\leq n\leq N_2} \frac{\chi(n)a_n e(\alpha n)}{n}\Bigg| = \max_{\beta\in \mathcal{A}} \Bigg|\sum_{N_1\leq n\leq N_2} \frac{\chi(n) a_ne(\beta n)}{n}\Bigg| +O\left(\frac{N_2}{R}\right). 
 $$
 Therefore, it follows from Minkowski's inequality that 
 \begin{equation}\label{eqTail1}
 \begin{aligned}
 \Bigg(\sum_{\chi \in \mathcal{D}} \max_{\alpha\in [0,1)} \Bigg|\sum_{N_1\leq n\leq N_2} \frac{\chi(n) a_ne(\alpha n)}{n}\Bigg|^{2k}\Bigg)^{1/2k}
 & \leq  \Bigg(\sum_{\chi \in \mathcal{D}} \max_{\alpha \in \mathcal{A}} \Bigg|\sum_{N_1\leq n\leq N_2} \frac{\chi(n) a_ne(\alpha n)}{n}\Bigg|^{2k}\Bigg)^{1/2k}\\
 & \quad \quad + O\left(|\mathcal{D}|^{1/2k} N_2R^{-1}\right)\\
 & \leq  \Bigg(\sum_{\alpha \in \mathcal{A}}\sum_{\chi \in \mathcal{D}}  \Bigg|\sum_{N_1\leq  n\leq N_2} \frac{\chi(n)a_n e(\alpha n)}{n}\Bigg|^{2k}\Bigg)^{1/2k}\\
 & \quad \quad + O\left(|\mathcal{D}|^{1/2k} N_2R^{-1}\right).\\
 \end{aligned}
 \end{equation}
%Using the elementary inequality $|a+b|^{2k}\leq 2^{2k} (|a|^{2k}+|b|^{2k})$ we deduce that
The lemma follows upon noting that 
$$
\Bigg|\sum_{N_1\leq n\leq N_2}\frac{\chi(n) a_ne(\alpha n)}{n}\Bigg|^{2k}= \Bigg|\sum_{N_1^k\leq n\leq N_2^k} \frac{\chi(n)g_{N_1, N_2, k}(n, \alpha)}{n}\Bigg|^{2}.
$$
\end{proof}

\begin{proof}[Proof of Proposition \ref{Tail1}]   Let 
$$
k:=\begin{cases} 2 & \textup{ if } \sqrt{x}\leq N_1\leq x^{21/40}, \\ 
\lfloor \log x/\log N_1\rfloor & \textup{ if } x^{\ep} \leq N_1< \sqrt{x}.
\end{cases} 
$$
We observe that $2\leq k\leq 1/ \ep$ by our assumption on $N_1$, and that $|g_{N_1, N_2, k}(n, \alpha)|\leq d_k(n)$ for all $N_1, N_2$ and $\alpha$. Let $\delta=1/100$ and $\nu>0$ be a small parameter to be chosen. Using Lemma \ref{Reduction} with the choice $R=N_1^{1+\delta}$ together with the large sieve inequality \eqref{HBLargeSieve} and the easy inequality $|a+b|^{k}\leq 2^{k} (|a|^k+|b|^k)$ we obtain 
\begin{equation}\label{eqTail2}
\begin{aligned}
&\sum_{d\in \Fo} \max_{\alpha\in [0,1)} \left|\sum_{N_1\leq n\leq N_2} \frac{\chi_d(n) a_n e(\alpha n)}{n}\right|^{2k} \\
& \ll_{\ep} \sum_{\alpha \in \mathcal{A}}  \sum_{d\in \Fo} \left|\sum_{N_1^k\leq n\leq N_2^k} \frac{\chi_d(n)g_{N_1, N_2, k}(n, \alpha)}{n}\right|^{2}+ \frac{x}{N_1^{2k\delta}}\\
& \ll_{\ep,\nu} (N_2^kx)^{\nu} (x+ N_2^k) \sum_{\alpha \in \mathcal{A}}\sum_{\substack{N_1^k\leq n_1, n_2 \leq N_2^k\\ n_1n_2= \square}}\frac{|g_{N_1,N_2, k}(n_1, \alpha)g_{N_1, N_2, k}(n_2, \alpha)|}{n_1n_2}+\frac{x}{N_1^{2k\delta}}\\
& \ll_{\ep, \nu} x^{21/20+3\nu} N_1^{1+\delta} \sum_{\substack{N_1^k\leq n_1, n_2\leq N_2^k \\ n_1n_2= \square}}\frac{d_k(n_1)d_k(n_2)}{n_1n_2}+\frac{x}{N_1^{2k\delta}},
\end{aligned}
\end{equation}
since $N_2^k\ll_{\ep} x^{21/20}$ by our assumption on $k$ and $N_1$. Now, writing $n_1n_2=n^2$, we get 
\begin{equation}\label{DivisorSums1}
\sum_{\substack{N_1^k\leq n_1, n_2\leq N_2^k \\ n_1n_2= \square}}\frac{d_k(n_1)d_k(n_2)}{n_1n_2}\leq  \sum_{n\geq N_1^k} \frac{d_{2k}(n^2)}{n^2} \leq  N_1^{-k(1-\nu)}  \sum_{n=1}^{\infty} \frac{d_{2k}(n^2)}{n^{1+\nu}}\ll_{\ep, \nu} N_1^{-k(1-\nu)}.
\end{equation}
Now, since $k\geq 2$, then $N_1^{3k/2}\geq N_1^{k+1}\geq x$, by our definition of $k$, which implies that $N_1^k\geq x^{2/3}$. Thus, choosing $\nu$ to be suitably small, and combining this estimate with \eqref{eqTail2} gives
$$ \sum_{d\in \Fo} \max_{\alpha\in [0,1)} \left|\sum_{N_1\leq n\leq N_2} \frac{\chi_d(n)a_n e(\alpha n)}{n}\right|^{2k} \ll_{\ep, \nu} x^{21/20+3\nu} N_1^{1+\delta} x^{-2/3(1-\nu)} + \frac{x}{N_1^{2k\delta}} \ll_{\ep} x^{1-4\delta/3}.$$
Finally, the number of fundamental discriminants $d\in \Fo$ such that 
$$\max_{\alpha\in [0,1)} \left|\sum_{N_1\leq n\leq N_2} \frac{\chi_d(n) a_ne(\alpha n)}{n}\right|\geq \frac{1}{(\log N_1)^A}$$ 
is 
$$ \leq (\log N_1)^{2Ak} \sum_{d\in \Fo} \max_{\alpha\in [0,1)} \left|\sum_{N_1\leq n\leq N_2} \frac{\chi_d(n) a_ne(\alpha n)}{n}\right|^{2k} \ll_{\ep, A} x^{1-\delta}, $$
which completes the proof.
\end{proof}
\begin{proof}[Proof of Proposition \ref{Tail2}]
We proceed similarly to the proof of Proposition \ref{Tail1} but we choose $k=\lfloor \log x/(3\log N_1)\rfloor$ in this case. Then  we have 
$$1/(3\ep)-1 \leq k\leq (\log x\log_4x)/ (3C\log_2 x \log _3 x)
$$ by our assumption on $N_1$. Using Lemma \ref{Reduction} with $R=N_1^2$ together with the large sieve inequality \eqref{BabyLargeSieve} and the easy inequality $|a+b|^{k}\leq 2^{k} (|a|^k+|b|^k)$ we obtain
\begin{equation}\label{eqTail3}
\begin{aligned}
&\sum_{d\in \Fo} \max_{\alpha\in [0,1)} \left|\sum_{N_1\leq n\leq N_2} \frac{\chi_d(n)a_n e(\alpha n)}{n}\right|^{2k} \\
& \ll e^{O(k)}\sum_{\alpha \in \mathcal{A}}  \sum_{d\in \Fo} \left|\sum_{N_1^k\leq n\leq N_2^k} \frac{\chi_d(n)g_{N_1, N_2, k}(n, \alpha)}{n}\right|^{2}+ \frac{xe^{O(k)}}{N_1^{2k}}\\
& \ll x e^{O(k)} \sum_{\alpha \in \mathcal{A}}\sum_{\substack{N_1^k\leq n_1, n_2 \leq N_2^k\\ n_1n_2= \square}}\frac{|g_{N_1, N_2, k}(n_1, \alpha)g_{N_1, N_2, k}(n_2, \alpha)|}{n_1n_2}+\frac{xe^{O(k)}}{N_1^{2k}}\\
& \ll x N_1^{2} e^{O(k)}  \sum_{n\geq N_1^k} \frac{d_{2k}(n^2)}{n^2}+\frac{x}{N_1^{2k}}
\end{aligned}
\end{equation}
by \eqref{DivisorSums1}, since $|g_{N_1, N_2, k}(n, \alpha)|\leq d_k(n)$ for all $N_1, N_2$ and $\alpha$. To bound the sum over $n$, we shall use Rankin's trick. Choosing $\nu=\log_3 k/(2\log k)$ and using \eqref{Divisor1} we obtain
$$ 
\sum_{n\geq N_1^k} \frac{d_{2k}(n^2)}{n^2} \leq  N_1^{-k\nu}  \sum_{n=1}^{\infty} \frac{d_{2k}(n^2)}{n^{2-\nu}} \ll \exp\left(-\frac{k\log_3 k\log N_1}{2\log k} +O(k\log_2 k)\right).
$$
Inserting this estimate in \eqref{eqTail3}  we deduce that 
$$ 
\sum_{d\in \Fo} \max_{\alpha\in [0,1)} \left|\sum_{N_1\leq n\leq N_2} \frac{\chi_d(n)a_n e(\alpha n)}{n}\right|^{2k} 
 \ll x \exp\left(-\frac{k\log_3 k\log N_1}{2\log k} +O(k\log_2 k +\log N_1)\right).
$$
Therefore, the number of fundamental discriminants $d\in \Fo$ such that 
$$\max_{\alpha\in [0,1)} \left|\sum_{N_1\leq n\leq N_2} \frac{\chi_d(n)a_n e(\alpha n)}{n}\right|\geq \frac{1}{(\log N_1)^A}$$ 
is 
\begin{align*}
& \leq (\log N_1)^{2Ak} \sum_{d\in \Fo} \max_{\alpha\in [0,1)} \left|\sum_{N_1\leq n\leq N_2} \frac{\chi_d(n)a_n e(\alpha n)}{n}\right|^{2k} \\
&
\ll x \exp\left(-\frac{k\log_3 k\log N_1}{2\log k} + O(k\log_2 k +Ak \log_2 N_1+ \log N_1)\right)\\
& 
\ll x  \exp\left(-\frac{\log x \log_4 x}{10 \log_2 x}\right)
\end{align*}
by our assumption on $N_1$ and $k$, if $\ep$ is suitably small and $C$ is suitably large. This completes the proof.

\end{proof}
%Let $P(y, z)$ be the set of positive integers $n$ such that $p\mid n \implies y<p\leq z$. 
Using Propositions \ref{Tail1} and \ref{Tail2} reduces the proof of Theorem \ref{KeyResult} to studying the distribution of the maximum over $\alpha\in[0, 1)$ of the very short sum 
$$ \Bigg|\sum_{\substack{1\leq n\leq Y\\ P^{+}(n)> y}} \frac{\chi_d(n) h(n) e(\alpha n)}{n}\Bigg|.$$
To this end we shall use the large sieve to bound its $2k$-th moments for \emph{every} $k\leq (\log x)/(3\log Y)$.

\begin{pro}\label{TailUnsmooth} Let $h(n)$ be a completely multiplicative function such that $|h(n)|\leq 1$ for all $n$. Let $C>0$ be a suitably large constant and put $Y= \exp(C\log_2 x \log _3 x/\log_4 x)$. For any positive integer $2\leq k\leq \log x/(3\log Y)$ and real number $(k\log k)/10\leq y\leq k^2$ we have 
$$
\sum_{d\in \Fo} \max_{\alpha\in [0,1)} \left|\sum_{\substack{1\leq n\leq Y\\ P^{+}(n)> y}} \frac{\chi_d(n)h(n) e(\alpha n)}{n}\right|^{2k} \ll x\left(e^{O\left(\frac{ k\log_2k}{\log k}\right)} \left(\frac{2e^{2\gamma}k\log y}{ey}\right)^{k}+O\left(\frac{1}{(\log k)^{10k}}\right)\right).
$$
\end{pro}

\begin{proof}[Proof of Proposition \ref{TailUnsmooth}]
We first define
\begin{equation}\label{DefinitionSd}
S_{d}(y, Y)= \max_{\alpha \in [0, 1)}\left|\sum_{\substack{2\leq n\leq Y\\ P^{-}(n)> y}} \frac{\chi_d(n) h(n) e(\alpha n)}{n}\right|.
\end{equation}
Then we have 
$$ \sum_{\substack{1\leq n\leq Y\\ P^{+}(n)> y}} \frac{\chi_d(n) h(n)e(\alpha n)}{n}= \sum_{\substack{1\leq a\leq Y\\ P^{+}(a)\leq  y}} \frac{\chi_d(a)h(a)}{a} \sum_{\substack{1<b\leq Y/a\\ P^{-}(b)> y}}\frac{\chi_d(b)h(b) e(\alpha ab)}{b},$$
and hence 
$$ \max_{\alpha \in [0, 1)}\left|\sum_{\substack{1\leq n\leq Y\\ P^{+}(n)> y}} \frac{\chi_d(n) h(n) e(\alpha n)}{n}\right| \leq \sum_{\substack{1\leq a\leq Y\\ P^{+}(a)\leq  y}} \frac{S_d(y, Y/a)}{a}.
$$
Therefore, using H\"older's inequality we obtain 
\begin{equation}\label{HolderSmooth}
\begin{aligned}
\sum_{d\in \Fo} \max_{\alpha\in [0,1)} \left|\sum_{\substack{1\leq n\leq Y\\ P^{+}(n)> y}} \frac{\chi_d(n) h(n)e(\alpha n)}{n}\right|^{2k}
& \leq \sum_{d\in \Fo} \left(\sum_{P^{+}(a)\leq y }\frac{1}{a}\right)^{2k-1} \sum_{\substack{1\leq a\leq Y\\ P^{+}(a)\leq  y}}\frac{S_d(y, Y/a)^{2k}}{a}\\
& \leq \left(e^{\gamma} \log y +O(1)\right)^{2k-1}\sum_{\substack{1\leq a\leq Y\\ P^{+}(a)\leq  y}}\frac{1}{a} \sum_{d\in \Fo}S_d(y, Y/a)^{2k}.
\end{aligned}
\end{equation}
by Mertens' Theorem. We shall now bound the moments 
$$ \sum_{d\in \Fo}S_d(y, z)^{2k}, $$
uniformly in $2\leq y<z\leq Y$. 
To this end we split the inner sum over $n$ into two parts $y <n\leq N$, and $N< n\leq z$, where 
\begin{equation}\label{DefinitionCutN}
N= \min\big(z, \exp(C\log k \log_2k /\log_3 k)\big).
\end{equation}
Note that the second part will be empty unless $z>\exp(C\log k \log_2k /\log_3 k)$.  Using Minkowski's inequality we get 
\begin{equation}\label{MinkUnsmooth}
\begin{aligned}
\left(\sum_{d\in \Fo}S_d(y, z)^{2k}\right)^{1/2k} 
& \leq \Bigg(\sum_{d\in \Fo}\max_{\alpha\in [0,1)} \Bigg|\sum_{\substack{2\leq n\leq N\\ P^{-}(n)> y}} \frac{\chi_d(n) h(n) e(\alpha n)}{n}\Bigg|^{2k}\Bigg)^{1/2k}\\
& \quad \quad + \Bigg(\sum_{d\in \Fo}\max_{\alpha\in [0,1)} \Bigg|\sum_{\substack{N<n\leq z\\ P^{-}(n)> y}} \frac{\chi_d(n) h(n) e(\alpha n)}{n}\Bigg|^{2k}\Bigg)^{1/2k}.
\end{aligned}
\end{equation}
We start by bounding the first term. By Lemma \ref{Reduction} with $R:=\exp(2C\log k \log_2k /\log_3 k)$ and $a_n=h(n)$ if $P^{-}(n)>y$ and equals $0$ otherwise, we obtain
\begin{equation}\label{Reduction2}
\begin{aligned}
&\Bigg(\sum_{d\in \Fo} \max_{\alpha\in [0,1)} \Bigg|\sum_{\substack{2\leq n\leq N\\ P^{-}(n)>y}} \frac{\chi_d(n) h(n)e(\alpha n)}{n}\Bigg|^{2k}\Bigg)^{1/2k} \\
& \leq \Bigg(\sum_{\alpha \in \mathcal{A}}  \sum_{d\in \mathcal{F}(x)} \Bigg|\sum_{\substack{n\leq N^{k}\\ P^{-}(n)>y}}\frac{\chi_d(n)g_{2, N, k}(n, \alpha)}{n}\Bigg|^{2}\Bigg)^{1/2k}+ O\left(x^{1/2k}R^{-{1/2}}\right),
\end{aligned}
\end{equation}
where $\mathcal{A}=\{b/R : 1\leq b\leq R\}.$ We now proceed similarly to the proof of \eqref{eqTail3}. Using the large sieve inequality \eqref{BabyLargeSieve}, and noting that $N^k\leq x^{1/3}$ and  $|g_{2, N, k}(n, \alpha)|\leq \widetilde{d_k}(n)$ for all $\alpha$ and $N$ we derive
\begin{equation}\label{TailDivisorSum0}
\begin{aligned}
\sum_{d\in \mathcal{F}(x)} \Bigg|\sum_{\substack{n\leq N^{k}\\ P^{-}(n)>y}}\frac{\chi_d(n)g_{2, N, k}(n, \alpha)}{n}\Bigg|^{2} 
&\ll x \sum_{\substack{n_1, n_2 \leq N^{k}\\ P^{-}(n_1n_2)> y\\ n_1n_2= \square}}
\frac{\widetilde{d_k}(n_1)\widetilde{d_k}(n_2)}{n_1n_2}  \leq  x \sum_{\substack{P^{-}(n)> y}} \frac{\widetilde{d_{2k}}(n^2)}{n^2}\\
& \ll x e^{O(k\log\log y/\log y)} \left(\frac{2k}{ey\log y}\right)^k,
\end{aligned}
\end{equation}
where the last inequality follows from \eqref{MykExpression} and Proposition \ref{MomentsUnsmooth}. 
%Let $\{\X(n)\}_{n\ge 1}$ be Rademacher random multiplicative functions.  Then, it follows from Proposition \ref{MomentsUnsmooth} that 
%\begin{equation}\label{TailDivisorSum}
%\sum_{\substack{n>1\\ P^{-}(n)> y}} \frac{d_{2k}(n^2)}{n^2} = \ex\Bigg(\sum_{\substack{n>1\\ P^{-}(n)> y}} \frac{d_{2k}(n)\X(n)}{n}\Bigg)= \mathcal{M}_{y}(k) \leq e^{O(k\log\log y/\log y)} \left(\frac{2k}{ey\log y}\right)^k. 
%\end{equation}
Inserting this estimate in \eqref{Reduction2} implies that 
\begin{equation}\label{FirstTermHolder}
\begin{aligned}
& \Bigg(\sum_{d\in \Fo} \max_{\alpha\in [0,1)} \Bigg|\sum_{\substack{2\leq  n\leq N\\ P^{-}(n)>y}} \frac{\chi_d(n) h(n) e(\alpha n)}{n}\Bigg|^{2k}\Bigg)^{1/2k} \\
& \quad \leq x^{1/(2k)}\left(e^{O(\log_2k/\log k)} \left(\frac{2k}{ey\log y}\right)^{1/2} + O\left(\frac{1}{R^{1/2}}\right)\right).
\end{aligned}
\end{equation}

We now assume that $z>\exp(C\log k \log_2k /\log_3 k)$ and bound the second term on the right hand side of \eqref{MinkUnsmooth}. We shall
split the inner sum over $n$ into dyadic intervals. Let $J_1= \lfloor \log N/\log 2\rfloor, J_2= \lfloor \log z/\log 2\rfloor$, and define $t_{J_1}= N$, $t_{J_2+1}= z$, and $t_j:= 2^j$ for $J_1+1\leq j\leq J_2$. Using H\"older's inequality we obtain 
\begin{equation}\label{HoldLongSum}
\begin{aligned}
\Bigg|\sum_{\substack{N<n\leq z\\ P^{-}(n)> y}} \frac{\chi_d(n) h(n)e(\alpha n)}{n}\Bigg|^{2k} 
&= \Bigg|\sum_{J_1\leq j\leq J_2} \frac{1}{j^2} \Bigg(j^2 \sum_{\substack{t_j<n\leq t_{j+1}\\ P^{-}(n)> y}} \frac{\chi_d(n) h(n) e(\alpha n)}{n}\Bigg)\Bigg|^{2k}\\
& \leq  \left(\sum_{J_1\leq j\leq J_2} \frac{1}{j^{\frac{4k}{2k-1}}}\right)^{2k-1} \sum_{J_1\leq j\leq J_2} j^{4k} \Bigg|\sum_{\substack{t_j<n\leq t_{j+1}\\ P^{-}(n)> y}} \frac{\chi_d(n) h(n) e(\alpha n)}{n}\Bigg|^{2k}\\
& \leq \left(\frac{C_1}{\log N} \right)^{2k+1} \sum_{J_1\leq j\leq J_2} j^{4k} \Bigg|\sum_{\substack{t_j<n\leq t_{j+1}\\ P^{-}(n)> y}} \frac{\chi_d(n) h(n) e(\alpha n)}{n}\Bigg|^{2k}, 
\end{aligned}
\end{equation}
for some constant $C_1>0$. 
Therefore, this reduces the problem to bounding the following moments over dyadic intervals $[t_j, t_{j+1}]$: $$\sum_{d\in \Fo}\max_{\alpha\in [0,1)} \Bigg|\sum_{\substack{t_j<n\leq t_{j+1}\\ P^{-}(n)> y}} \frac{\chi_d(n) h(n) e(\alpha n)}{n}\Bigg|^{2k}.$$
Let $\mathcal{B}_j= \{b/4^j : 1\le b\leq 4^j \}.$ By the easy inequality $|a+b|^k\leq 2^k(|a|^k+|b|^k)$ together with Lemma \ref{Reduction} with $R=4^j$ and the same choice of $a_n$ as before, we derive
\begin{equation}\label{LSieveUnsmooth}
\begin{aligned}
&\sum_{d\in \Fo} \max_{\alpha\in [0,1)} \Bigg|\sum_{\substack{t_j<n\leq t_{j+1}\\ P^{-}(n)> y}} \frac{\chi_d(n) h(n) e(\alpha n)}{n}\Bigg|^{2k} \\
& \ll e^{O(k)}\sum_{\alpha \in \mathcal{B}_j}  \sum_{d\in \mathcal{F}(x)} \Bigg|\sum_{\substack{t_j^k<n\leq t_{j+1}^k\\ P^{-}(n)> y}} \frac{\chi_d(n)g_{t_j, t_{j+1}, k}(n, \alpha)}{n}\Bigg|^{2}+ \frac{xe^{O(k)}}{2^{2jk}},
\end{aligned}
\end{equation}
since $t_j\asymp t_{j+1} \asymp 2^j$. 
Furthermore, similarly to the proof of \eqref{eqTail3}, it follows from the large sieve inequality \eqref{BabyLargeSieve} that the main term on the right hand side of \eqref{LSieveUnsmooth} is
\begin{equation}\label{EstimateDivisors14}
\ll x e^{O(k)} \sum_{\alpha \in \mathcal{B}_j}\sum_{\substack{t_j^k<n_1, n_2 \leq t_{j+1}^k\\ P^{-}(n_1n_2)> y\\ n_1n_2= \square}}\frac{|g_{t_j, t_{j+1}, k}(n_1, \alpha)g_{t_j, t_{j+1}, k}(n_2, \alpha)|}{n_1n_2}  \ll e^{O(k)}  4^j x \sum_{n\geq t_j^{k}} \frac{d_{2k}(n^2)}{n^2},
\end{equation}
 since $t_{j+1}^k\leq z^k\leq Y^k\leq  x^{1/3}$. We now put $\nu=\log_3 k/(2\log k)$ and use \eqref{Divisor1} to get
\begin{equation}\label{EstimateDivisors15}
\sum_{n\geq t_j^{k}} \frac{d_{2k}(n^2)}{n^2} \leq  t_j^{-k\nu}  \sum_{n=1}^{\infty} \frac{d_{2k}(n^2)}{n^{2-\nu}} \ll \exp\left(-\frac{j k\log_3 k}{4\log k} +O(k\log_2 k)\right).
\end{equation}
Inserting these estimates in \eqref{LSieveUnsmooth} gives
$$ \sum_{d\in \Fo} \max_{\alpha\in [0,1)} \Bigg|\sum_{\substack{t_j<n\leq t_{j+1}\\ P^{-}(n)> y}} \frac{\chi_d(n) h(n) e(\alpha n)}{n}\Bigg|^{2k} \ll x\exp\left(-\frac{j k\log_3 k}{5\log k} +O(k\log_2 k)\right).$$
Using this last estimate together with \eqref{HoldLongSum}, and noting that $j^4 \leq \exp(j\log_3 k/(20 \log k))$ for all $j\geq  J_1$ by our choice of $N$ if $C$ is suitably large, we derive
\begin{align*}
\sum_{d\in \Fo} \max_{\alpha\in [0,1)} \Bigg|\sum_{\substack{N<n\leq z\\ P^{-}(n)> y}} \frac{\chi_d(n) h(n) e(\alpha n)}{n}\Bigg|^{2k} 
&\ll x e^{O(k\log_2 k)} \sum_{J_1\leq j\leq J_2} \exp\left(-\frac{j k\log_3 k}{20\log k} \right) \\
& \ll x \exp\left(-\frac{(\log N) k\log_3 k}{20\log k} +O(k\log_2k)\right)\\
& \ll x \exp\left(-\frac{C}{40}k\log_2k\right).
\end{align*} 
Combining this estimate with \eqref{MinkUnsmooth} and \eqref{FirstTermHolder} we obtain 
$$ \sum_{d\in \Fo}S_d(y, z)^{2k} \ll x \left(e^{O(\log_2k/\log k)} \left(\frac{2k}{ey\log y}\right)^{1/2}+O\left(\frac{1}{(\log k)^{C/80}}\right)\right)^{2k},$$
uniformly for $y<z\leq Y$. 
Therefore, inserting this estimate in \eqref{HolderSmooth} we deduce 
\begin{align*}
& \sum_{d\in \Fo} \max_{\alpha\in [0,1)} \Bigg|\sum_{\substack{n\leq Y\\ P^{+}(n)> y}} \frac{\chi_d(n) h(n) e(\alpha n)}{n}\Bigg|^{2k}
\\
& \ll x\left(e^{O\left(\frac{\log_2k}{\log k}\right)} \left(\frac{2e^{2\gamma}k\log y}{ey}\right)^{1/2}+O\left(\frac{1}{(\log k)^{C/80-1}}\right)\right)^{2k}.
\end{align*}
The result follows upon choosing $C$ to be suitably large, and  using the basic inequality\footnote{This inequality simply follows by considering the two cases $a\leq \sqrt{b}$ and $a\geq \sqrt{b}$.} $(a+b)^{2m} \leq (a^{2m}+ \sqrt{b}^{2m})(1+\sqrt{b})^{2m}\leq (a^{2m}+b^m) e^{2m\sqrt{b}}$, which  is valid for all real numbers $a, b>0$ and positive integers $m$. 
 \end{proof}

 We end this section by deducing Theorem \ref{KeyResult} from Propositions \ref{Tail1}, \ref{Tail2}, and \ref{TailUnsmooth}. 
 
 \begin{proof}[Proof of Theorem \ref{KeyResult}]

 Let $Y:= \exp(C\log_2 x \log _3 x/\log_4 x)$ for some suitably large constant $C>0$. Let $L_1:=\lfloor \log Y /\log 2\rfloor$, $L_2:= \lfloor \log Z/\log 2\rfloor$, and define $s_{L_1}:= Y$, $s_{L_2+1}:= Z$, and $s_{\ell}:=2^{\ell}$ for $L_1+1\leq \ell\leq L_2$.
Then we have
  \begin{equation}\label{PolyaSplit}
\begin{aligned}
\max_{\alpha \in [0, 1)} \Bigg|\sum_{\substack{1\leq n\leq Z\\P^+(n)>y}} \frac{\chi_d(n) h(n) e(\alpha n)}{n}\Bigg| 
&  \leq  \max_{\alpha \in [0, 1)} \Bigg|\sum_{\substack{1\leq n\leq Y\\ P^{+}(n)> y}} \frac{\chi_d(n) h(n) e(\alpha n)}{n}\Bigg| \\
& + \sum_{L_1\leq \ell \leq L_2} \max_{\alpha \in [0, 1)} \Bigg|\sum_{\substack{s_{\ell}\leq n\leq s_{\ell+1}\\ P^+(n)>y}} \frac{\chi_d(n) h(n) e(\alpha n)}{n}\Bigg|.
\end{aligned}
\end{equation}

%Then it follows from \eqref{Polya}, \eqref{PolyaSplit} and \eqref{CorBGGK} that for all $d\in \F$ and all $ 3\leq y \leq Y$ we have 
%\begin{equation}\label{Bound_mchi}
%\begin{aligned}
%m(\chi_d)  \leq & \log y +e^{-\gamma} \log 2 + 2 e^{-\gamma} \max_{\alpha \in [0, 1)} \left|\sum_{\substack{1\leq n\leq Y\\ P^{+}(n)> y}} \frac{\chi_d(n) e(\alpha n)}{n}\right|  \\
%& + 2 e^{-\gamma} \sum_{\frac{\log Y}{\log 2}\leq \ell \leq \frac{\log Z}{\log 2}} \max_{\alpha \in [0, 1)} \left|\sum_{2^{\ell}\leq n\leq  2^{\ell+1}} \frac{\chi_d(n) e(\alpha n)}{n}\right| + O\left(\frac{\log\log y}{\log y}\right). 
%\end{aligned}
%\end{equation}
Using Propositions \ref{Tail1} and \ref{Tail2} with $A=2$ and $a_n=h(n)$ if $P^+(n)>y$ and equals $0$ otherwise, we deduce that 
\begin{equation}\label{BoundTailLarge}
\sum_{L_1\leq \ell \leq L_2} \max_{\alpha \in [0, 1)} \Bigg|\sum_{\substack{s_{\ell}\leq n\leq  s_{\ell+1}\\ P^+(n)>y}} \frac{\chi_d(n) h(n) e(\alpha n)}{n}\Bigg| 
\ll  
\sum_{L_1\leq \ell \leq L_2} \frac{1}{\ell^2} \ll \frac{\log_4 x}{\log_2 x\log_3 x},
\end{equation}
for all fundamental discriminants $|d|\leq x$ except for a set $\mathcal{E}(x)$ of size 
$$ |\mathcal{E}(x)| \ll L_2 x \exp\left(-\frac{\log x \log_4 x}{10 \log_2 x}\right)  \ll x \exp\left(-\frac{\log x \log_4 x}{20 \log_2 x}\right).$$
%Therefore, we deduce that 
%\begin{equation}\label{UpperPhiLast}
%\Phi_x(\tau) \leq  \widetilde{\Phi}_{x, %y}\left(A-\frac{c_0\log \tau }{\tau }\right) + O\left( %\exp\left(-\frac{\log x \log_4 x}{20 \log_2 %x}\right)\right), 
%\end{equation}
%where $c_0>0$ is an absolute constant, and 
%$$ \widetilde{\Phi}_{x, y} (B):= \frac{1}{|\F|}\Bigg|\Bigg\{d\in \F :  \max_{\alpha \in [0, 1)} \Big|\sum_{\substack{1\leq n\leq Y\\ P^{+}(n)> y}} \frac{\chi_d(n) e(\alpha n)}{n}\Big|>e^{\gamma} B\Bigg\}\Bigg|.$$
For $B>0$ we let $\widetilde{\Psi}_{x, y} (B)$ be the proportion of fundamental discriminants $|d|\leq x$ such that 
$$ \max_{\alpha \in [0, 1)} \Bigg|\sum_{\substack{1\leq n\leq Y\\ P^{+}(n)> y}} \frac{\chi_d(n) h(n) e(\alpha n)}{n}\Bigg| > e^{\gamma} B.$$
Let $k$ be a positive integer satisfying the assumptions of Proposition  \ref{TailUnsmooth}. Then, it follows from this result that 
\begin{align*}
\widetilde{\Psi}_{x, y} (B) &\leq (e^{\gamma} B)^{-2k} \frac{1}{|\Fo|} \sum_{d\in \Fo} \max_{\alpha\in [0,1)} \Bigg|\sum_{\substack{1\leq n\leq Y\\ P^{+}(n)> y}} \frac{\chi_d(n) h(n) e(\alpha n)}{n}\Bigg|^{2k} \\
& \ll e^{O\left(\frac{ k\log_2k}{\log k}\right)} \left(\frac{2k\log y}{e B^2y}\right)^{k}+O\left(\frac{1}{(e^{\gamma} B(\log k)^5)^{2k}}\right).
\end{align*}
We now assume that $1/(\log y)^2\leq B\leq \sqrt{20}$ and choose $k= \lfloor (B^2 y)/(2 \log y) \rfloor$. This gives 
\begin{equation}\label{BoundMidTailDistribution}
\widetilde{\Psi}_{x, y} (B) \ll \exp\left(-\frac{B^2y}{2 \log y} \left(1+O\left(\frac{\log_2 y}{\log y}\right)\right)\right). 
\end{equation}
Combining this estimate with \eqref{PolyaSplit} and \eqref{BoundTailLarge} and choosing  $B=A-C_0\log_4 x/(\log_2 x\log_3 x)$ for some suitably large constant $C_0$ completes the proof.  
\end{proof}

%%%%%%%%%%%%%%%%%%%%%%%%%%%%%%%%%%%%%%%%%%%%%%%%%%%%%%%%%%%%%%%%%%%%%%%%%%%%%%%%%%%%%%%%%%%%%%%%%%%%%%%%%%%%%%%%%%%%%%%%%%%%%%%%%%%%%%%%%%%%%

\section{The distribution of the tail of the P\'olya Fourier series for prime discriminants: Proof of Theorem \ref{KeyResultPrimes}}

In order to prove Theorem \ref{KeyResultPrimes} we follow the same lines of the proof of Theorem \ref{KeyResult} but we replace the large sieve inequality \eqref{BabyLargeSieve} by the following large sieve inequality for prime discriminants, which is a special case of Lemma 9 of Montgomery and Vaughan \cite{MV79} (see also Lemma 1 of \cite{BaMo}). 

\begin{lem} [Lemma 9 of \cite{MV79}]\label{LargeSieveMV}
Let $x, N$ be real numbers such that $x\geq 2$ and  $2\leq N\leq x^{1/3}$. Then for arbitrary complex numbers $a_1, \dots, a_N$ we have 
$$
\sum_{p\le x} \left|\sum_{n\leq N} a_n \psi_p(n)\right|^2 \ll \frac{x}{\log x} \sum_{\substack{m, n \leq N\\ mn= \square}} |a_ma_n|. %\log\left(\frac{2x}{N^2}\right)^{-1} \sum_{s} \mu(s)^2 \left(\sum_{m} |a_{sm^2}|\right)^{2}.
$$
\end{lem}
Since the number of primes up to $x$ is smaller by a factor of size $\log x$ compared to the number of fundamental discriminants up to $x$, it suffices to establish the analogue of Proposition \ref{TailUnsmooth} for prime discriminants. Indeed, the savings in Propositions \ref{Tail1} and \ref{Tail2} are much larger than $\log x$, and hence we can use these results in this setting as well by simply embedding the set of primes $p\leq x$ in the set of fundamental discriminants $d$ with $|d|\leq x$.

\begin{pro}\label{TailUnsmoothPrimes} Let $h(n)$ be a completely multiplicative function such that $|h(n)|\leq 1$ for all $n$. Let $Y= \exp(C\log_2 x \log _3 x/\log_4 x)$, for some suitably large constant $C$. Then, for any positive integer $2\leq k\leq \log x/(3\log Y)$ and real number $(k\log k)/10\leq y\leq k^2$ we have 
$$
\sum_{ p\leq x} \max_{\alpha\in [0,1)} \Bigg|\sum_{\substack{1\leq n\leq Y\\ P^{+}(n)> y}} \frac{\psi_p(n)h(n) e(\alpha n)}{n}\Bigg|^{2k} \ll \pi(x)\left(e^{O\left(\frac{ k\log_2k}{\log k}\right)} \left(\frac{2e^{2\gamma}k\log y}{ey}\right)^{k}+O\left(\frac{1}{(\log k)^{10k}}\right)\right).
$$
\end{pro}
\begin{proof} 
We shall closely follow the proof of Proposition \ref{TailUnsmooth} and only indicates where the main changes occur. First we define
\begin{equation}\label{DefinitionSdPrimes}
\widetilde{S}_{p}(y, Y)= \max_{\alpha \in [0, 1)}\Bigg|\sum_{\substack{2\leq n\leq Y\\ P^{-}(n)> y}} \frac{\psi_p(n) h(n) e(\alpha n)}{n}\Bigg|.
\end{equation}
Then, similarly to \eqref{HolderSmooth} we have
\begin{equation}\label{HolderSmoothPrimes}
\sum_{p\leq x} \max_{\alpha\in [0,1)} \Bigg|\sum_{\substack{n\leq Y\\ P^{+}(n)> y}} \frac{\psi_p(n) h(n) e(\alpha n)}{n}\Bigg|^{2k}
 \leq \left(e^{\gamma} \log y +O(1)\right)^{2k-1}\sum_{\substack{a\leq Y\\ P^{+}(a)\leq  y}}\frac{1}{a} \sum_{p\leq x}\widetilde{S}_p(y, Y/a)^{2k}.
\end{equation}
Let $y<z\leq Y$ be a real number and $N$ be defined by \eqref{DefinitionCutN}. Using Minkowski's inequality as in \eqref{MinkUnsmooth} we get
\begin{equation}\label{MinkUnsmoothPrimes}
\begin{aligned}
\left(\sum_{p\leq x}\widetilde{S}_p(y, z)^{2k}\right)^{1/2k} 
& \leq \Bigg(\sum_{p\leq x}\max_{\alpha\in [0,1)} \Bigg|\sum_{\substack{2\leq n\leq N\\ P^{-}(n)> y}} \frac{\psi_p(n) h(n) e(\alpha n)}{n}\Bigg|^{2k}\Bigg)^{1/2k}\\
& \quad \quad + \Bigg(\sum_{p\leq x}\max_{\alpha\in [0,1)} \Bigg|\sum_{\substack{N<n\leq z\\ P^{-}(n)> y}} \frac{\psi_p(n) h(n) e(\alpha n)}{n}\Bigg|^{2k}\Bigg)^{1/2k}.
\end{aligned}
\end{equation}
 
We start by bounding the first term. Let $R=\exp(2C\log k \log_2k /\log_3 k)$. Using the same argument leading to \eqref{FirstTermHolder} and replacing the large sieve inequality \eqref{BabyLargeSieve} by Lemma \ref{LargeSieveMV}  we obtain
\begin{equation}\label{FirstTermHolderPrimes}
\begin{aligned}
&\Bigg(\sum_{p\leq x} \max_{\alpha\in [0,1)} \Bigg|\sum_{\substack{2\leq  n\leq N\\ P^{-}(n)>y}} \frac{\psi_p(n) h(n) e(\alpha n)}{n}\Bigg|^{2k}\Bigg)^{1/2k} \\
%& \leq \Bigg(\sum_{\alpha \in \mathcal{A}}  \sum_{p\leq x} \Bigg|\sum_{\substack{n\leq N^{k}\\ P^{-}(n)>y}}\frac{\psi_p(n)g_{2, N, k}(n, \alpha)}{n}\Bigg|^{2}\Bigg)^{1/2k}+ O\left(\pi(x)^{1/2k}R^{-{1/2}}\right)\\
& \leq \pi(x)^{1/(2k)}\left(e^{O(\log_2k/\log k)} \left(\frac{2k}{ey\log y}\right)^{1/2} + O\left(\frac{1}{R^{1/2}}\right)\right).
\end{aligned}
\end{equation}

We now assume that $z>\exp(C\log k \log_2k /\log_3 k)$ and bound the second term on the right hand side of \eqref{MinkUnsmoothPrimes}. By \eqref{HoldLongSum} we have
\begin{equation}\label{HoldLongSumPrimes}
\Bigg|\sum_{\substack{N<n\leq z\\ P^{-}(n)> y}} \frac{\psi_p(n) h(n) e(\alpha n)}{n}\Bigg|^{2k} 
 \leq \left(\frac{C_1}{\log N} \right)^{2k+1} \sum_{J_1\leq j\leq J_2} j^{4k} \Bigg|\sum_{\substack{t_j<n\leq t_{j+1}\\ P^{-}(n)> y}} \frac{\psi_p(n) h(n) e(\alpha n)}{n}\Bigg|^{2k}, 
\end{equation}
for some constant $C_1>0$, where $J_1= \lfloor \log N/\log 2\rfloor, J_2= \lfloor \log z/\log 2\rfloor$, $t_{J_1}= N$, $t_{J_2+1}= z$, and $t_j= 2^j$ for $J_1+1\leq j\leq J_2$. As before we let $\mathcal{B}_j= \{b/4^j : 1\le b\leq 4^j \}.$ Combining Lemmas \ref{Reduction} and \ref{LargeSieveMV} with \eqref{EstimateDivisors14} and \eqref{EstimateDivisors15} we deduce that
\begin{align*}
&\sum_{p\leq x} \max_{\alpha\in [0,1)} \Bigg|\sum_{\substack{t_j<n\leq t_{j+1}\\ P^{-}(n)> y}} \frac{\psi_p(n) h(n) e(\alpha n)}{n}\Bigg|^{2k} \\
& \ll e^{O(k)}\sum_{\alpha \in \mathcal{B}_j}  \sum_{p\leq x} \Bigg|\sum_{\substack{t_j^k<n\leq t_{j+1}^k\\ P^{-}(n)> y}} \frac{\psi_p(n)g_{t_j, t_{j+1}, k}(n, \alpha)}{n}\Bigg|^{2}+ \frac{\pi(x)e^{O(k)}}{2^{2jk}}\\
& \ll \pi(x) \exp\left(-\frac{j k\log_3 k}{5\log k} +O(k\log_2 k)\right).
\end{align*}
Using this last estimate together with \eqref{HoldLongSumPrimes} we derive 
$$
\sum_{p\leq x} \max_{\alpha\in [0,1)} \Bigg|\sum_{\substack{N<n\leq z\\ P^{-}(n)> y}} \frac{\psi_p(n) h(n)e(\alpha n)}{n}\Bigg|^{2k} 
\ll \pi(x) \exp\left(-\frac{C_1}{40}k\log_2k\right).
$$
We have now established the key estimates over prime discriminants. Thus we continue the proof along the exact same lines as the proof of Proposition \ref{TailUnsmooth}, by combining this last estimate with \eqref{HolderSmoothPrimes}, \eqref{MinkUnsmoothPrimes} and \eqref{FirstTermHolderPrimes}. This yields the desired result. 
\end{proof}
We end this section by proving Theorem \ref{KeyResultPrimes}.
\begin{proof}[Proof of Theorem \ref{KeyResultPrimes}]
As before we let  $Y= \exp(C\log_2 x \log _3 x/\log_4 x)$ for some suitably large constant $C>0$, and put $L_1=\lfloor \log Y /\log 2\rfloor$, $L_2= \lfloor \log Z/\log 2\rfloor$, $s_{L_1}= Y$, $s_{L_2+1}= Z$, and $s_{\ell}=2^{\ell}$ for $L_1+1\leq \ell\leq L_2$. By \eqref{PolyaSplit} we have
\begin{align*}
\max_{\alpha \in [0, 1)} \Bigg|\sum_{\substack{1\leq n\leq Z\\P^+(n)>y}} \frac{\psi_p(n) h(n) e(\alpha n)}{n}\Bigg| 
&  \leq  \max_{\alpha \in [0, 1)} \Bigg|\sum_{\substack{1\leq n\leq Y\\ P^{+}(n)> y}} \frac{\psi_p(n) h(n) e(\alpha n)}{n}\Bigg| \\
& + \sum_{L_1\leq \ell \leq L_2} \max_{\alpha \in [0, 1)} \Bigg|\sum_{\substack{s_{\ell}\leq n\leq s_{\ell+1}\\ P^+(n)>y}} \frac{\psi_p(n) h(n) e(\alpha n)}{n}\Bigg|.
\end{align*}
By embedding the set of primes $3\leq p\leq x$ into the set of  fundamental discriminants $|d|\leq x$ (since for any such prime, $p$ or $-p$ is a fundamental discriminant), it follows from \eqref{BoundTailLarge} that 
\begin{equation}\label{BoundTailLarge2}
\sum_{L_1\leq \ell \leq L_2} \max_{\alpha \in [0, 1)} \Bigg|\sum_{\substack{s_{\ell}\leq n\leq s_{\ell+1}\\ P^+(n)>y}} \frac{\psi_p(n) h(n) e(\alpha n)}{n}\Bigg|
 \ll \frac{\log_4 x}{\log_2 x\log_3 x},
\end{equation}
for all primes $p\leq x$ except for a set $\mathcal{E}_1(x)$ of size 
$ |\mathcal{E}_1(x)| \ll  x \exp\left(-\frac{\log x \log_4 x}{20 \log_2 x}\right).$ The result follows along the same exact lines of the proof of Theorem \ref{KeyResult}, by replacing Proposition \ref{TailUnsmooth} by Proposition \ref{TailUnsmoothPrimes}. 
\end{proof}

 %%%%%%%%%%%%%%%%%%%%%%%%%%%%%%%%%%%%%%%%%%%%%%%%%%%%%%%%%%%%%%%%%%%%%%%%%%%%%%%%%%%%%%%%%%%%%%%%%%%%%%%%%%%%%%%%%%%%%%%%%%%%%%%%%%%%%%%%%%%%%%%%%%%%%%%%%%%%%%%%%%%%%%%%%%%%%%%%%%%%%%%%%%%%%%
 
 \section{The distribution of $L(1,\psi_p)$: proof of Theorem \ref{LamzouriL1chi}}

 \subsection{Proof of the estimate \eqref{EstimateDistribL1psi}}
%In this case the proof is similar to that of the lower bound of Theorem \ref{Main}. First, note that
%$$ m(\psi_p) \geq e^{-\gamma} L(1, \psi_p),$$
%for all primes $p\equiv 3 \bmod 4$ by \eqref{LowerBoundML1chi}. Therefore, \eqref{LowerBPrimeDis} follows from the following result, which is a consequence of Theorem 1.2 of \cite{LaP}.

In the range $2\leq \tau \leq (\log_2x)/2-2\log_3x$, we shall use the work of the author \cite{La17} where we proved an asymptotic formula for complex moments of $L(1,\psi_p)$ involving a secondary term coming from a possible Laudau-Siegel exceptional discriminant. Indeed, it follows from Theorem 1.2 of \cite{La17} that for all real numbers $2\leq k\leq \sqrt{\log x}/(\log_2 x)^2$ we have 
\begin{equation}\label{MomentsL1chiPrimes}
\frac{2}{\textup{Li}(x)}\sum_{\substack{p\leq x\\ p\equiv 3 \bmod 4}} L(1, \psi_p)^k= \ex\left(L(1, \X)^k\right) + E_1+ O\left(\exp\left(-\frac{\sqrt{\log x}}{10\log\log x}\right)\right),
\end{equation}
where 
$$ |E_1|\leq \big|\ex\big(\X(|d_1|)L(1, \X)^k\big)\big|, $$
and where $d_1$ is the possible ``exceptional'' discriminant\footnote{By the Landau-Page Theorem (see Chapter 20 of \cite{Da}), there is at most one square-free integer $d_1$ such that $|d_1|\leq \exp(\sqrt{\log x})$ and $L(s, \chi_{d_1})$ has a zero in the region 
$
\re(s)>1-c/\sqrt{\log x},
$
 for some positive constant $c$. If it exists, we refer to such $d_1$ as the exceptional discriminant in the range $|d_1|\leq \exp(\sqrt{\log x})$.} with $|d_1|\leq \exp(\sqrt{\log x})$, and 
 $$ L(1, \X):=\prod_{q \text{ prime}}\left(1-\frac{\X(q)}{q}\right)^{-1},$$ 
 with the $\X(q)$ being I.I.D. random variables taking the values $\pm 1$ with probability $1/2$. We cannot rule out that the term $E_1$ is of the same size as the main term, but we will prove that it will not heavily affect the size of the $k$-th moment of $L(1, \psi_p)$ if $k$ is large. Indeed, we observe that 
\begin{equation}\label{BoundErrorSiegel}
\frac{\ex\big(\X(|d_1|)L(1, \X)^k\big)}{\ex\big(L(1, \X)^k\big)}= \prod_{q\mid d_1} \frac{\left(1-\frac{1}{q}\right)^{-k}-\left(1+\frac{1}{q}\right)^{-k}}{\left(1-\frac{1}{q}\right)^{-k}+\left(1+\frac{1}{q}\right)^{-k}}= \prod_{q\mid d_1} \frac{1-\delta_q}{1+\delta_q},
\end{equation}
 where $\delta_q= \left(1-\frac{2}{q+1}\right)^k$. Let $\ep>0$ be a suitably small constant. Let $q_1$ be the largest prime factor of $d_1$. Since $d_1$ is square-free and $d_1>(\log x)^{10/\ep}$, if $x$ is large enough, by Siegel's Theorem, we must have $q_1>(\log |d_1|)/2>(5/\ep)\log\log x$, since otherwise $|d_1|<\prod_{q\leq (\log |d_1|)/2}q= |d_1|^{1/2+o(1)}$ by the Prime Number Theorem, which is a contradiction. Inserting this estimate in \eqref{BoundErrorSiegel} gives  
\begin{equation}\label{BoundErrorSiegel2}
 0< \frac{\ex\left(\X(|d_1|)L(1, \X)^k\right)}{\ex\left(L(1, \X)^k\right)} \leq 1-\delta_{q_1}=1 -\left(1-\frac{2}{q_1+1}\right)^k \leq 1-\exp\left(-\ep\frac{k}{\log k}\right),
\end{equation}
 for all real numbers $2\leq k\leq \sqrt{\log x}/(\log_2 x)^2.$ Now, it follows from Proposition 1.2 of \cite{La10} that 
\begin{equation}\label{EstimationMomentL1Random}
\ex\left(L(1, \X)^k\right)= \exp\left(k\log_2k+ k\gamma+ \frac{k}{\log k} \left(B_0-1+O\left(\frac{1}{\log k}\right)\right)\right).
\end{equation}
Combining this estimate with \eqref{MomentsL1chiPrimes} and \eqref{BoundErrorSiegel2} we obtain
 %Now, by Equations (2.4) and (2.5) of \cite{LaP} we have 
% $$ \ex\left(\X(|q_1|)L(1, \X)^k\right) =\sum_{n=1}^{\infty} \frac{d_k(|d_1|n^2)}{|q_1|n^2} \leq \frac{d_k(|q_1|)}{|q_1|} \ex\left(L(1, \X)^k\right).$$
% since $d_k(mn)\leq d_k(m)d_k(n)$. Now, it follows from Corollary 1.36 of \cite{Norton} that $d_k(n)\ll n^{O(\log k/\log\log n)}$ in the range $k=o(\log n)$. Using this estimate together with the fact that $|q_1|\geq (\log x)/(\log_2 x)^4$ (see Chapter 20 of \cite{Da}) we deduce that there exists a positive constant $C_3$ such that if $k\leq (\log_2 x)^{C_3}$, then $d_k(|q_1|)\ll |q_1|^{1/3}.$ Inserting this estimate in \eqref{MomentsL1chiPrimes} and using  we deduce that
 \begin{equation}\label{AsymptoticL1chiPrime}
 \frac{2}{\pi(x)}\sum_{\substack{p\leq x\\ p\equiv 3 \bmod 4}} L(1, \psi_p)^k
= \exp\left(k\log_2k+ k\gamma+ \frac{k}{\log k} \left(B_0-1+O\left(\ep\right)\right)\right).
 \end{equation}
Finally, the estimate for the distribution function \eqref{EstimateDistribL1psi} follows from the proof of  Theorem 0.1 of \cite{La10} (which holds for general random models of this type), with the choice $\tau= \log k+B_0$. This completes the proof of \eqref{EstimateDistribL1psi} in the case $a=3$. We also note that the case $a=1$ is similar, since one can derive the same estimate as \eqref{MomentsL1chiPrimes} over the primes $p\equiv 1\bmod 4$ using the method of \cite{La17}.

We now show how to obtain the same estimate for the proportion of primes $p\equiv a \bmod 4$ such that $L(1, \psi_p\chi_{-3})>(2e^{\gamma}/3)\tau$.  
First, by a slight adaptation of the proof of Theorem 1.2  of \cite{La17} we get
\begin{equation}\label{MomentsL1chiPrimes2}
\frac{2}{\textup{Li}(x)}\sum_{\substack{p\leq x\\ p\equiv a \bmod 4}} L(1, \psi_p\chi_{-3})^k= \ex\left(L(1, \X\chi_{-3})^k\right) + E_2+ O\left(\exp\left(-\frac{\sqrt{\log x}}{10\log\log x}\right)\right),
\end{equation}
where 
$$ |E_2|\leq \big|\ex\big(\X(|d_1|)L(1, \X\chi_{-3})^k\big)\big|, $$
and where $d_1$ is the exceptional discriminant (if it exists) with $|d_1|\leq \exp\left(\sqrt{\log x}\right)$, and 
$$ L(1, \X\chi_{-3}):=\prod_{q \text{ prime}}\left(1-\frac{\X(q)\chi_{-3}(q)}{q}\right)^{-1}.$$
By the independence of the $\X(q)$'s we observe that 
$$  
\ex\left(L(1, \X\chi_{-3})^k\right)= \prod_{q\neq 3} \ex\left(\left(1-\frac{\X(q)}{q}\right)^{-k}\right)= \left(\frac{2}{3}\right)^{k}\left(\frac{2}{1+2^{-k}} \right)\ex\left(L(1, \X)^k\right).
$$
Furthermore, if we fix a suitably small constant $\ep>0$, then a similar argument leading to \eqref{BoundErrorSiegel2} gives 
\begin{equation}\label{BoundErrorSiegel3}
|E_2|\leq \left(1-\exp\left(-\ep\frac{k}{\log k}\right)\right)\ex\left(L(1, \X\chi_{-3})^k\right),
\end{equation}
if $x$ is large enough. Thus by \eqref{EstimationMomentL1Random} we deduce that 
$$ \frac{2}{\pi(x)}\sum_{\substack{p\leq x\\ p\equiv a \bmod 4}} \left(\frac{3}{2}L(1, \psi_p\chi_{-3})\right)^k=\exp\left(k\log_2k+ k\gamma+ \frac{k}{\log k} \left(B_0-1+O\left(\ep\right)\right)\right).$$
Finally, the result follows from the proof of Theorem 0.1 of \cite{La10}. 

\begin{rem}\label{PreciseLowerConditional2}
 We observe that assuming GRH,  Holmin, Jones, Kurlberg, McLeman and Petersen \cite{HJKMP} established the asymptotic formula \eqref{MomentsL1chiPrimes} without the term $E_1$ in the larger range $k\leq (\log x)/(50 (\log\log x)^2)$. This justifies Remark \ref{PreciseLowerConditional}.

\end{rem} 

%%%%%%%%%%%%%%%%%%%%%%%%%%%%%%%%%%%%%%

\subsection{Proof of \eqref{EstimateDistribL1psi2}}
 
In the range $(\log_2x)/2-2\log_3x\leq \tau\leq \log_2 x-\log_3x-C_2$, we need to use a different argument since asymptotic formulas for very large moments of $L(1, \psi_p)$ are not known, due to the lack of strong unconditional bounds on character sums over primes. In this case, our strategy is to use Theorem \ref{KeyResultPrimes} to truncate $L(1, \psi_p)$ over $y$-friable integers, and then control the behaviour of $\psi_p(q)$ over the  primes $q\leq y$. To this end we establish the following lemma which follows from Bombieri's proof  of Linnik's Theorem (see Chapter 5 of \cite{Bo}). 
 \begin{lem}\label{SignsLegendreLong}
Let $\{\ep_q\}_{q \text{ prime}}$ be a sequence of $\pm 1$, and  $a\in \{1, 3\}$. Let $x$ be large and $3\leq y\leq c_0 \log x$ be a real number, where $c_0$ is a suitably small constant. Let $\mathcal{P}(x, y, a, \{\ep_q\})$ be the set of primes $p\equiv a \bmod 4$ with $p\leq x$ such that $\psi_p(q)=\ep_q$ for all primes $q \leq y$. Then we have 
$$ 
|\mathcal{P}(x, y, a, \{\ep_q\})|\gg \pi(x) \exp\left(-3y +O\left(\frac{y}{\log y}\right)\right).
$$
 
\end{lem}
\begin{proof}
First recall that if $p\equiv 3 \bmod 4$ then $\psi_p(2)=1$ if $p\equiv 7 \bmod 8$, and equals $-1$ if $p\equiv 3 \bmod 8$. Similarly, if $p\equiv 1 \bmod 4$ then $\psi_p(2)=1$ if $p\equiv 1 \bmod 8$, and $\psi_p(2)=-1$ if $p\equiv 5 \bmod 8$. Let $3\leq q\leq y$ be a prime number. By the law of quadratic reciprocity, there exists a residue class $b_q \bmod q$ such that if $p\equiv b_q\bmod q$ then $\psi_p(q)= \ep_q$. We now define $Q:= 8\prod_{3\leq q\leq y} q$. Then, by the Chinese Remainder Theorem there exists a residue class $b\bmod Q$ such that if $p\equiv b\bmod Q$ then $p\equiv a \bmod 4$ and $\psi_p(q)=\ep_q$ for all primes $q \leq y$. By Bombieri's proof of Linnik's Theorem (see Chapter 6 of \cite{Bo}) we deduce that there exists a small constant $c>0$ such that if $x>Q^{1/c}$ then
$$ |\mathcal{P}(x, y, a, \{\ep_q\})| \geq \pi(x; Q, b) \gg \frac{\pi(x)}{Q^{3}}.$$
The result follows upon noting that $Q=\exp(y(1+O(1/\log y)))$ by the Prime Number Theorem.

\end{proof}

%\begin{pro}\label{L1chiPrimeLinnik}
%There exists positive constants $c_1, c_2$ such that in the range $\tau\leq  \log_2x-\log_3x-c_1$ we have 
%$$ \frac{1}{\pi(x)} \left|\left\{p\leq x :  \ m(\psi_p)>\tau\right\}\right|\gg  \exp\left(-c_2 \tau e^{\tau}\right).$$
%\end{pro}

\begin{proof}[Proof of the lower bound \eqref{EstimateDistribL1psi2}]  
%Let $\chi$ be a non-principal character modulo $q$. By partial summation and the P\'olya-Vinogradov inequality we have uniformly
%for $\alpha\in [0, 1)$
%\begin{equation}\label{TruncPolyaTAIL}
%\sum_{n>Z}\frac{\chi(n)e(n\alpha)}{n} \ll %\frac{\sqrt{q}\log q}{Z}.
%\end{equation} 
Let $Z=x^{21/40}$. By \eqref{PolyaUniformTail} we have $$L(1, \psi_p)= \sum_{n\leq Z}\frac{\psi_p(n)}{n} +O(1),$$
%\sum_{\substack{n\leq x^{21/40}\\ P^{+}(n)\leq y_1}} \frac{\psi_p(n)}{n}= \sum_{\substack{n\geq 1\\ P^{+}(n)\leq y_1}} \frac{\psi_p(n)}{n} +O(x^{-\delta}),$$
for all primes $p\leq x$.
Let $c_0$ be the constant in Lemma \ref{SignsLegendreLong}, and $3\leq y\le c_0 \log x$ be a parameter to be chosen. Define $\mathcal{P}(y)$ to be the set of primes $p\leq x$ such that $p\equiv a\bmod 4$, $\psi_p(q)=1$ for all primes $q\leq y$ and 
\begin{equation}\label{AssumptionPY}
 \Bigg|\sum_{\substack{1\leq n\leq Z\\P^{+}(n)> y_1}} \frac{\psi_p(n) }{n}\Bigg|\leq e^{\gamma}, 
\end{equation}
where $y_1:=10 y\log y.$ Combining Lemma \ref{SignsLegendreLong}  with Theorem \ref{KeyResultPrimes}  yields 
\begin{equation}\label{LowerBPY}
|\mathcal{P}(y)| \gg \pi(x) \exp\left(-4y\right).
\end{equation}
%Let $p\in \mathcal{P}(y)$, and choose $\alpha=1/2$.  
%Then we deduce from \eqref{PolyaPrimes} and \eqref{AssumptionPY} that 
%$$
% m(\psi_p)\geq \frac{e^{-\gamma}}{2} |F(1/2, \psi_p)|+O(p^{-\delta})\geq e^{-\gamma} \left|\sum_{\substack{n\geq 1\\P^{+}(n)\leq y_1}} \frac{\psi_p(n) (1-(-1)^n)}{n}\right| +O(1),
%$$since $\psi_p$ is odd. 
%Furthermore, we have 
%$$ 
%\sum_{\substack{n\geq 1\\P^{+}(n)\leq y_1}} \frac{\psi_p(n) (1-(-1)^n)}{n}= 2\sum_{\substack{n\geq 1, \  n \text{ odd}\\P^{+}(n)\leq y_1}} \frac{\psi_p(n)}{n}= (2-\psi_p(2)) \sum_{\substack{n\geq 1, \\P^{+}(n)\leq y_1}} \frac{\psi_p(n)}{n}.
%$$
On the other hand, by \eqref{PolyaUniformTail3} and \eqref{AssumptionPY} we deduce that for every prime $p\in \mathcal{P}(y)$ we have
\begin{align*}
L(1, \psi_p)&=\sum_{\substack{1\leq n\leq Z\\ P^{+}(n)\leq y_1}} \frac{\psi_p(n)}{n} +O(1)= \sum_{\substack{n\geq 1\\ P^{+}(n)\leq y_1}} \frac{\psi_p(n)}{n} +O(1)\\
&= \prod_{q\leq y} 
\left(1-\frac{1}{q}\right)^{-1} \exp\left(\sum_{y<q\leq y_1}\frac{\psi_p(q)}{q}+O\left(\frac{1}{y}\right)\right)+O(1)\\
& \geq e^{\gamma}\log y \cdot \exp\left(-\sum_{y<q\leq y_1}\frac{1}{q}\right)\left(1+O\left(\frac{1}{\log y}\right)\right)+O(1)\\
%& \geq e^{\gamma} \log y \frac{\log y}{\log(y (\log y)^2)} \left(1+O\left(\frac{1}{\log y}\right)\right)\\
& \geq e^{\gamma} \log y - e^{\gamma} \log\log y +O(1), 
\end{align*}
by Mertens' Theorem. Combining the above estimates and choosing $y= C\tau e^{\tau}$ for some large constant $C$, implies that  $L(1, \psi_p)> e^{\gamma}\tau$ whenever $p\in \mathcal{P}(y)$. Appealing to \eqref{LowerBPY} completes the proof.

\end{proof}

We end this section by proving the upper bound of \eqref{EstimateDistribL1psi2} which we deduce from Theorem \eqref{KeyResultPrimes}. 

\begin{proof}[Proof of the upper bound of \eqref{EstimateDistribL1psi2}]
Let $Z=x^{21/40}$ and $\delta=1/100$. Let $y=e^{\tau-B}$, where $B$ is a parameter to be chosen later. First by \eqref{PolyaUniformTail} we observe that 
\begin{equation}\label{UpperL1psiTrunc}
\begin{aligned}
L(1, \psi_p) & = \sum_{n\leq Z}\frac{\psi_p(n)}{n} +O(x^{-\delta})  \leq  \sum_{\substack{n\geq 1\\ P^{+}(n)\leq y}} \frac{1}{n}+ \Bigg|\sum_{\substack{1\leq n\leq Z\\ P^{+}(n)>y}} \frac{\psi_p(n)}{n}\Bigg|+O(x^{-\delta}) \\
&\leq e^{\gamma}\tau- e^{\gamma} B+ \Bigg|\sum_{\substack{1\leq n\leq Z\\ P^{+}(n)>y}} \frac{\psi_p(n)}{n}\Bigg|+O\left(\frac{1}{ \tau}\right).
\end{aligned}
\end{equation}
since $\prod_{p\leq y}(1-1/p)^{-1}= e^{\gamma}\log y+O(1/\log y)$ by the Prime Number Theorem. Therefore, there exists a positive constant $c$ such that the proportion of primes $p\leq x$ with $p\equiv a \bmod 4$ and such that $L(1, \psi_p)\geq e^{\gamma} \tau$ is bounded by the proportion of primes $p\leq x$ such that 
$$  \Bigg|\sum_{\substack{1\leq n\leq Z\\ P^{+}(n)>y}} \frac{\psi_p(n)}{n}\Bigg| \geq e^{\gamma} \left(B- \frac{c}{\tau}\right).$$
Thus, appealing to Theorem \ref{KeyResultPrimes} with $B= 2+c/\tau$ completes the proof of \eqref{EstimateDistribL1psi2} for $L(1, \psi_p)$. Finally, the analogous result for $L(1, \psi_p\chi_{-3})$ follows along the same lines upon taking $h(n)=\chi_{-3}(n)$ in Theorem \ref{KeyResultPrimes} and noting that 
$$\Bigg|\sum_{\substack{1\leq n\leq Z\\ P^{+}(n)\leq y}} \frac{\psi_p(n)\chi_{-3}(n)}{n}\Bigg|\leq \prod_{\substack{q\neq 3 \\ q\leq y}}\left(1-\frac{1}{q}\right)^{-1}= \frac{2e^{\gamma}}{3} \log y + O\left(\frac{1}{\log y}\right), 
$$ 
by the Prime Number Theorem.
\end{proof}

\section{Positivity of sums of the Legendre symbol : Proof of Theorems \ref{PositivityLegendreTheorem} and \ref{NegativityLegendreTheorem} }

Let $Z=x^{21/40}$ and $\delta=1/100$. Let $p\leq x$ be a prime number such that $p \equiv 3 \bmod 4$. Since $\psi_p$ is odd and
%, and as before denote by $\psi_p(\cdot):=\left(\frac{\cdot}{p}\right)$ the Legendre symbol modulo $p$. 
$\tau(\psi_p)=i\sqrt{p}$, then by \eqref{PolyaFourier}  we have for any $\alpha\in (0, 1)$ 
\begin{equation}\label{PolyaFourier2}
\sum_{n\leq \alpha p } \psi_p(n)= \frac{\sqrt{p}}{\pi}\sum_{1\leq n\leq Z} \frac{\psi_p(n) (1-\cos(2\pi n\alpha))}{n} +O\left(p^{1/2-1/40}\log p\right).
\end{equation}
We shall now focus on the sum 
$$
F(\alpha, p):=\sum_{1\leq n\leq Z} \frac{\psi_p(n) (1-\cos(2\pi n\alpha))}{n}.
$$
Our strategy for proving Theorems \ref{PositivityLegendreTheorem} and \ref{NegativityLegendreTheorem} consists in using Theorem \ref{KeyResultPrimes} to bound the part of this sum over non-friable integers uniformly over $\alpha\in [0,1)$, for all primes $p\leq x$ except for a small set of exceptions, and then prescribe the signs of $\psi_p(q)$ for the small primes $q$. 
%\begin{equation}\label{BoundFalpha}
%\begin{aligned} 
%F(\alpha, p)-\sum_{\substack{1\leq |n|\leq Y\\ P^{+}(n)\leq y}} \frac{\psi_p(n) (1-e(-n\alpha))}{n}& =  \left(\sum_{\substack{1\leq |n|\leq Y\\ P^{+}(n)\leq y}} + \sum_{\substack{1\leq |n|\leq Y\\ P^{+}(n)\leq y}} + \sum_{Y< |n|\leq Z} \right)\frac{\psi_p(n) (1-e(-n\alpha))}{n}\\
%& \geq 2\sum_{\substack{n\leq Y\\ P^{+}(n)\leq y}} \frac{\psi_p(n) (1-\cos(2\pi n\alpha))}{n}- 4 \max_{\beta\in [0, 1)} \left|\sum_{\substack{1\leq n\leq Y\\ P^{+}(n)> y}} \frac{\psi_p(n) e(\beta n)}{n}\right| \\
%& \quad \quad - 4 \sum_{\frac{\log Y}{\log 2}\leq \ell \leq \frac{\log Z}{\log 2}} \max_{\beta \in [0, 1)} \left|\sum_{2^{\ell}\leq n\leq  2^{\ell+1}} \frac{\psi_p(n) e(\beta n)}{n}\right|,
%\end{aligned}
%\end{equation}

\begin{proof}[Proof of Theorem \ref{PositivityLegendreTheorem}]
We start by proving \eqref{PreciseLargeLambda}. Let $Z=x^{21/40}$ and $2\leq y\leq \log x$ be a parameter to be chosen. 
 Let $\mathcal{E}_1(x)$ be the set of primes $\sqrt{x}\leq p\leq x$ such that $p\equiv 3\bmod 4$ and    
 $$ 
\max_{\alpha\in [0, 1)}\Bigg|\sum_{\substack{1\leq n\leq Z\\P^{+}(n)>y}} \frac{\psi_p(n) (1-\cos(2\pi n\alpha))}{n}\Bigg|>  2e^{\gamma}.
$$
Then it follows from Theorem \ref{KeyResultPrimes}  that 
\begin{equation}\label{SizeExpSetPrimes}
|\mathcal{E}_1(x)| \ll \pi(x)\exp\left(-\frac{y}{3\log y}\right).
\end{equation}
We now define $\mathcal{A}(y)$ to be the set of primes $\sqrt{x}\leq p\leq x$ such that $p\equiv 3\bmod 4$ and $\psi_p(q)=1$ for all primes $q\leq y_0$, where $y_0:= y/(20\log y)$. Moreover, we put $ \mathcal{D}(y):= \mathcal{A}(y)\setminus \mathcal{E}_1(x).$ By Lemma \ref{SignsLegendreLong} and the estimate \eqref{SizeExpSetPrimes} we have $|\mathcal{D}(y)|\gg \pi(x)\exp(-y/(4\log y))$. Moreover, if $p\in \mathcal{D}(y)$ then it follows from \eqref{PolyaUniformTail3} that for all $\alpha\in (0, 1)$ we have 
\begin{equation}\label{BoundFalpha2}
\begin{aligned}
F(\alpha, p)
 &=\sum_{\substack{1\leq n\leq Z\\P^{+}(n)\leq y}} \frac{\psi_p(n)(1-\cos(2\pi n\alpha))}{n} +O(1)= 
 \sum_{\substack{n\geq 1\\P^{+}(n)\leq y}} \frac{\psi_p(n)(1-\cos(2\pi n\alpha))}{n} +O(1)\\
 &= \sum_{\substack{n\geq 1\\P^{+}(n)\leq y_0}} \frac{1-\cos(2\pi n\alpha)}{n} +O(\log_2 y)=\sum_{\substack{n\geq 1\\P^{+}(n)\leq y}} \frac{1-\cos(2\pi n\alpha)}{n} +O(\log_2 y),
 \end{aligned}
\end{equation}
since 
$$ \sum_{\substack{n\geq 1\\y_0< P^{+}(n)\leq y}} \frac{1}{n}= \prod_{p\leq y}\left(1-\frac{1}{p}\right)^{-1}- \prod_{p\leq y_0}\left(1-\frac{1}{p}\right)^{-1}\ll \log_2 y, $$
by Mertens' Theorem. Furthermore, by Lemma 3.4 of  \cite{BGGK} we have 
\begin{equation}\label{BoundCosFriable}
\Bigg|\sum_{\substack{n\geq 1\\P^{+}(n)\leq y}} \frac{\cos(2\pi n\alpha)}{n}\Bigg| \leq \Bigg|\sum_{\substack{n\geq 1\\P^{+}(n)\leq y}}\frac{e(n\alpha)}{n}\Bigg|= \sum_{\substack{n\leq 1/||\alpha||\\ P^{+}(n)\leq y}} \frac1n+O(1),
\end{equation}
where $||\cdot||$ denotes the distance to the nearest integer. We now  write $||\alpha||= y^{-u}$ for some $u>0$. By Lemma 3.3 of \cite{BGGK} we get 
$$ \sum_{\substack{n\leq 1/||\alpha||\\ P^{+}(n)\leq y}}\frac1n = \sum_{\substack{n\leq y^u\\ P^{+}(n)\leq y}}\frac1n = (\log y)\int_0^u \rho(t) dt +O(1), $$
where $\rho$ is the Dickman?de Bruijn function, defined by $\rho(t)$ for $0\leq t\leq 1$, and 
$t\rho'(t)= - \rho(t-1)$ for $t> 1$. Combining this estimate with \eqref{BoundFalpha2} and \eqref{BoundCosFriable} we obtain 
\begin{equation}\label{LowerBFalphap} F(\alpha, p) \geq \log y\left(e^{\gamma} - \int_0^u \rho(t) dt\right) +O(\log_2 y), \
\end{equation}
for all $p\in \mathcal{D}(y)$ and all $\alpha\in (0,1)$, where $u= -\log ||\alpha||/\log y$. Note that $\int_0^u\rho(t) dt$ is increasing in $u$ and that $\int_0^{\infty} \rho(t) dt =e^{\gamma}$. Moreover, one has the estimate 
\begin{equation}\label{EstimateRho1} e^{\gamma} - \int_0^u \rho(t) dt = \int_{u}^{\infty} \rho(t) dt= \frac{1}{u^{u(1+o(1))}},
\end{equation}
which follows from the standard estimate $\rho(t)= t^{-t(1+o(1))}$. Let $C$ be a suitably large constant, and $u_0$ be the solution of the equation 
\begin{equation}\label{EstimateRho2}\int_{u_0}^{\infty} \rho(t) dt = \frac{C\log_2 y}{\log y}.
\end{equation}
Then, it follows from \eqref{LowerBFalphap} that for  all $p\in \mathcal{D}(y)$ and $||\alpha||> y^{-u_0}$ we have $F(\alpha, p) \geq 10$ if $C$ is suitably large, and hence
$$ \sum_{n\leq \alpha p } \psi_p(n)>\sqrt{p} $$
 by \eqref{PolyaFourier2} if $x$ is large enough. To finish the proof, we choose $y$ such that $T=y^{u_0}/2$, which implies that $\lambda(p)>1-1/T$ for all $p\in \mathcal{D}(y)$.  Moreover, it follows from the estimates \eqref{EstimateRho1} and \eqref{EstimateRho2} that $u_0= (1+o(1))\log_2 y/\log_3 y$, and hence we get 
$$ y= \exp\left(\frac{(1+o(1))\log T\log_3 T}{\log_2 T}\right) , $$
as desired.

\end{proof}

%We can perform this last step using Lemma \ref{SignsLegendreLong}, but when the parameter $T$ (in the statements of Theorems \ref{PositivityLegendreTheorem} and \ref{NegativityLegendreTheorem}) is small.

In order to prove \eqref{PreciseSmallLambda} for small $T$, we require a more precise estimate than the one provided by Lemma \ref{SignsLegendreLong}. To this end we prove the following result, which gives an asymptotic formula for the cardinality of the set $\mathcal{P}(x, y, a, \{\ep_q\})$ in Lemma \ref{SignsLegendreLong}, in the smaller range  $3\leq y \leq \log\log x$. 

%This is an unconditional version of Lemma ? of Montgomery \cite{Mo}, who proved a similar result in the extended range $y\ll \log x$ assuming GRH. The proof is analogous to Lemma ? of \cite{Mo} (the only main difference is the use of the Siegel-Walfisz Theorem instead of GRH) but we include it for the sake of completeness. 

\begin{lem}\label{SignsLegendre}
Let $\{\ep_q\}_{q \text{ prime}}$ be a sequence of $\pm 1$, and  $a\in \{1, 3\}$. Let $x$ be large and $3\leq y\leq \log\log x$ be a real number. Let $\mathcal{P}(x, y, a, \{\ep_q\})$ be the set in Lemma \ref{SignsLegendreLong}. Then we have 
$$ 
|\mathcal{P}(x, y, a, \{\ep_q\})|=\frac{\pi(x)}{2^{\pi(y)+1}}+ O\left(xe^{-c\sqrt{\log x}}\right),
$$
for some positive constant $c$. 
\end{lem}
\begin{proof}
%We only prove this estimate when $a=1$ since the proof in the case $a=3$ is similar. 
Let $Q=\prod_{q\leq y} q$. Note that by the Prime Number Theorem and our assumption we have $Q=e^{y+o(y)}\leq  (\log x)^{1+o(1)}.$ We observe that for a prime $y< p\leq x$ we have
\begin{equation}\label{DetectingPrimesSigns}
\frac{1}{2^{\pi(y)+1}} (1+\chi_{-4}(a)\chi_{-4}(p))\prod_{q\leq y} \left(1+ \ep_q \psi_p(q)\right)=\begin{cases} 
1 & \text{ if } p\in \mathcal{P}(x, y, a, \{\ep_q\}), \\ 0 & \text{ otherwise,}\end{cases}
\end{equation}
where $\chi_{-4}$ is the non-principal character modulo $4$. We extend the sequence $\ep_q$ multiplicatively to all square-free numbers $\ell$ by letting $\ep_{\ell}=\prod_{q\mid \ell} \ep_q$. Therefore, expanding the product on the left hand side of \eqref{DetectingPrimesSigns} we deduce that
\begin{equation}\label{ExpandProductCharacters}
\begin{aligned}
|\mathcal{P}(x, y, a, \{\ep_q\})| &= \frac{1}{2^{\pi(y)+1}} \sum_{y< p\leq x} \left(1+\chi_{-4}(a)\chi_{-4}(p)\right)\prod_{q\leq y} \left(1+ \ep_q \psi_p(q)\right)+O(y)\\
&= \frac{1}{2^{\pi(y)+1}} \sum_{p\leq x} (1+\chi_{-4}(a)\chi_{-4}(p))\sum_{\ell \mid Q}  \ep_{\ell}\psi_p(\ell)+O(y)\\
& = \frac{1}{2^{\pi(y)+1}} \left(\sum_{\ell \mid Q} \ep_{\ell} \sum_{p\leq x} \left(\frac{\ell}{p}\right)+ \chi_{-4}(a) \sum_{\ell \mid Q} \ep_{\ell} \sum_{p\leq x} \chi_{-4}(p)\left(\frac{\ell}{p}\right)\right)+O(y).
\end{aligned}
\end{equation}
If $\ell\neq 1$, then the law of quadratic reciprocity implies that $\xi_1=\left(\frac{\ell}{\cdot}\right)$ is a non-principal character of conductor $\ell$ or $4\ell$. Similarly, for all $\ell $ the character $\xi_2= \chi_{-4}\xi_1$ is a non-principal character of conductor $4\ell$. Furthermore, note that $4\ell \leq 4 Q \leq (\log x)^2$ if $x$ is large enough. Thus, it follows from Corollary 11.18 of \cite{MVBook}  that for $j=1$ if $\ell\neq1$ and $\ell\mid Q$, and for $j=2$ for all $\ell\mid Q$ we have 
\begin{equation}\label{SiegelWalfisz}
\sum_{p\leq x} \xi_{j}(p) \ll x\exp\left(-c\sqrt{\log x}\right).
\end{equation}
Inserting these bounds in \eqref{ExpandProductCharacters} completes the proof.

\end{proof}

\begin{proof}[Proof of Theorem \ref{NegativityLegendreTheorem}] 
We start by proving \eqref{PreciseSmallLambda}. Let $Z=x^{21/40}$ and $\delta=1/100$. By \eqref{PolyaFourier2} we have 
$$ \sum_{n\leq \alpha p } \psi_p(n)= \frac{\sqrt{p}}{\pi}F(\alpha, p) +O\left(p^{1/2-\delta}\right),$$
for all primes $\sqrt{x}\le p\leq x$ and all $\alpha\in (0,1)$. Let $3\leq y\leq (\log_2 x)(\log_3 x)^2$ be a parameter to be chosen, and $\mathcal{E}_2(x)$ be the set of primes $\sqrt{x} \leq p\leq x$ such that $p\equiv 3\bmod 4$ and    
$$
\max_{\alpha\in [0, 1)}\Bigg|\sum_{\substack{1\leq n\leq Z\\P^{+}(n)> y}} \frac{\psi_p(n) (1-\cos(2\pi n\alpha))}{n}\Bigg|>  \frac{2e^{\gamma}}{\log y}.
$$
By Theorem \ref{KeyResultPrimes} we have 
\begin{equation}\label{SizeExpSetPrimes2}
|\mathcal{E}_2(x)| \ll \pi(x)\exp\left(-\frac{y}{3(\log y)^3}\right).
\end{equation}
We will use the same choice as Montgomery \cite{Mo} for the values of $\psi_p(q)$ for small $q$. Let $h(n)$ be the completely multiplicative function such that $h(q)=\chi_{-3}(q)$ for all primes $q\neq 3$ and $h(3)=-1$.  Montgomery showed in \cite{Mo} that for all $\alpha\in (0,1)$ such that $||\alpha||<1/3$ we have 
\begin{equation}\label{MontgomeryTrick}
U(\alpha):=\sum_{n=1}^{\infty} \frac{h(n)\cos(2\pi n\alpha)}{n}>\frac{\pi}{8\sqrt{3}}.
\end{equation}
Let $y_0:=y/(3\log y)^2$ and $3\leq H\leq y_0$ be a parameter to be chosen. Let $\mathcal{B}(y, H)$ be the set of primes $\sqrt{x}\leq p\leq x$ such that $p\equiv 3\bmod 4$ and
$$
\psi_p(q)=\begin{cases} h(q) & \text{ if } q\leq H,\\
-1 & \text{ if } H<q\leq y_0.
\end{cases}
$$ We define $ \mathcal{T}(y):= \mathcal{B}(y, H)\setminus \mathcal{E}_2(x).$ Then we have $|\mathcal{T}(y)|\gg \pi(x)/2^{\pi(y_0)}$ by Lemma \ref{SignsLegendre} and the estimate \eqref{SizeExpSetPrimes2}. 
We now let $p$ be a prime in $\mathcal{T}(y)$. By \eqref{PolyaUniformTail3} and our assumption on $p$ we get 
\begin{equation}\label{Falphap}
\begin{aligned}
F(\alpha, p)&= \sum_{\substack{1\leq n\leq Z\\P^{+}(n)\leq y}} \frac{\psi_p(n)(1-\cos(2\pi n\alpha))}{n} +O\left(\frac{1}{\log y}\right)\\
&= \sum_{\substack{n\geq 1\\P^{+}(n)\leq y }} \frac{\psi_p(n)(1-\cos(2\pi n\alpha))}{n} +O\left(\frac{1}{\log y}\right).
\end{aligned}
\end{equation}
Furthermore, we have
\begin{equation}\label{L1chi3}
\begin{aligned}
\sum_{\substack{n\geq 1\\P^{+}(n)\leq y}} \frac{\psi_p(n)}{n}&= \prod_{q\leq y} \left(1-\frac{\psi_p(q)}{q}\right)^{-1}\\
&= \frac{3}{4} \prod_{q\leq H}\left(1-\frac{\chi_{-3}(q)}{q}\right)^{-1} \prod_{H< q\leq y_0} \left(1+\frac{1}{q}\right)^{-1}\prod_{y_0<q\leq y} \left(1-\frac{\psi_p(q)}{q}\right)^{-1}.
\end{aligned}
\end{equation}
By the Prime Number Theorem in arithmetic progressions we have 
$$ \sum_{q\leq t} \chi_{-3}(q) \ll t \exp\left(-c\sqrt{\log t}\right), $$
for some constant $c>0$. Therefore, by partial summation we obtain
$$ \sum_{q>H} \frac{\chi_{-3}(q)}{q} \ll \exp\left(-\frac{c}{2}\sqrt{\log H}\right).$$ 
This implies that 
\begin{align*}
\prod_{q\leq H}\left(1-\frac{\chi_{-3}(q)}{q}\right)^{-1} &=  L(1, \chi_{-3}) + O\left(\exp\left(-\frac{c}{2}\sqrt{\log H}\right)\right)\\
&= \frac{\pi}{3\sqrt{3}}+O\left(\exp\left(-\frac{c}{2}\sqrt{\log H}\right)\right).
\end{align*}
Inserting this estimate in \eqref{L1chi3} and using Mertens' Theorem we deduce that 
$$
0< \sum_{\substack{n\geq 1\\P^{+}(n)\leq y}} \frac{\psi_p(n)}{n}\leq  \frac{\pi}{4\sqrt{3}}\frac{\log H\log y}{(\log y_0)^2} \left(1+O\left(\frac{1}{\log H}\right)\right).
$$
We now choose $H= \sqrt{y}/(C_1(\log y)^2)$ where $C_1$ is a suitably large constant. This gives
\begin{equation}\label{L1chi32}
\sum_{\substack{n\geq 1\\P^{+}(n)\leq y}} \frac{\psi_p(n)}{n}\leq  \frac{\pi}{8\sqrt{3}} -\frac{\log C_1}{\log y}.
\end{equation}
On the other hand, by Parseval's identity we have 
\begin{align*}
 \int_0^{1}\Bigg|\sum_{\substack{n\geq 1\\P^{+}(n)\leq y}} \frac{\psi_p(n) \cos(2\pi n\alpha)}{n} - U(\alpha)\Bigg|^2d\alpha & = \sum_{\substack{n\geq 1\\ P^+(n)\leq y}} \frac{(\psi_p(n)- h(n))^2}{n^2} + \sum_{\substack{n\geq 1\\ P^+(n)> y}} \frac{h(n)^2}{n^2}\\
 & \ll \sum_{\substack{n\geq 1\\P^{+}(n)>H}}\frac{1}{n^2} \ll \frac{\log y}{\sqrt{y}}, 
 \end{align*}
since 
$$ 
\sum_{\substack{n\geq 1\\P^{+}(n)>H}}\frac{1}{n^2}= \frac{\pi^2}{6}- 
\sum_{\substack{n\geq 1\\P^{+}(n)\leq H}}\frac{1}{n^2} = \frac{\pi^2}{6}\left(1-\prod_{q>H} \left(1-\frac{1}{q^2}\right)\right)\ll \sum_{q>H} \frac{1}{q^2} \ll \frac{1}{H\log H}.
$$
Let $\mathcal{S}_p$ be the set of $\alpha\in (0, 1)$ such that 
$$ \Bigg|\sum_{\substack{n\geq 1\\P^{+}(n)\leq y}} \frac{\psi_p(n) \cos(2\pi n\alpha)}{n} - U(\alpha)\Bigg| > \frac{1}{\log y}.$$
Then 
$$ \mu(\mathcal{S}_p)\leq (\log y)^2 \int_0^{1}\Bigg|\sum_{\substack{n\geq 1\\P^{+}(n)\leq y}} \frac{\psi_p(n) \cos(2\pi n\alpha)}{n} - U(\alpha)\Bigg|^2d\alpha \ll \frac{(\log y)^3}{\sqrt{y}}.
$$
Thus, recalling the definitions of the sets $\mathcal{T}(y)$ and $\mathcal{S}_p$ and using the estimate  \eqref{MontgomeryTrick}  we deduce that if $p\in \mathcal{T}(y)$ and $\alpha\in (0,1/3)\cup (2/3, 1)\setminus \mathcal{S}_p$ then 
$$ \sum_{\substack{n\geq 1\\P^{+}(n)\leq y}}\frac{\psi_p(n) \cos(2\pi n\alpha)}{n}> \frac{\pi}{8\sqrt{3}}-\frac{1}{\log y}.$$
Combining this estimate with \eqref{Falphap} and \eqref{L1chi32} gives
\begin{align*}
F(\alpha, p)=  \sum_{\substack{n\geq 1\\P^{+}(n)\leq y}} \frac{\psi_p(n)}{n}- \sum_{\substack{n\geq 1\\P^{+}(n)\leq y}} \frac{\psi_p(n)\cos(2\pi n\alpha)}{n}+O\left(\frac{1}{\log y}\right) \leq -\frac{C_2 }{\log y}, 
\end{align*}
for some positive constant $C_2$ if $C_1$ is suitably large. This in turn implies that 
$$ \sum_{n\leq \alpha p } \psi_p(n)<-\frac{\sqrt{p}}{\log_3 p}. $$
by \eqref{PolyaFourier2}, if $x$ is large enough. Hence for $p\in \mathcal{T}(y)$ we have 
$$ \lambda(p)\leq \frac{1}{3}+ \mu(\mathcal{S}_p)\leq \frac{1}{3}+ C_3\frac{(\log y)^3}{\sqrt{y}}, $$
for some absolute constant $C_3>0$. Choosing $y= C_4 T^2 (\log T)^6$ for some suitably large constant $C_4$ completes the proof of \eqref{PreciseSmallLambda}. 

To prove \eqref{SmallLambda} we follow the exact same lines, and replace Lemma \ref{SignsLegendre} by Lemma \ref{SignsLegendreLong}. In this case we make the following choices of the parameters: $y_0=y/(3\log y)^3\leq c_0\log x$ (where $c_0$ is the constant in Lemma \ref{SignsLegendreLong}), $H=\sqrt{y}/(C_5(\log y)^3)$, and $y= C_6 T^2 (\log T)^{8}$ for some suitably large constants $C_5$ and $C_6$. This completes the proof.

\end{proof}

%%%%%%%%%%%%%%%%%%%%%%%%%%%%%%%%%%%%%%%%%%%%%%%%%%%%%%%%%%%%%%%%%%%%%%%%%%%%%%%%%%%%%%%%%%%%%%%%%%%%%%%%%%%%%%%%%%%%%%%%%%%%%%%%%%%%%%%%%%%%%%%%%%%%%%%%%%%%%%%%%%%%%%%%%%%%%%%%%

\end{document}